\documentclass[12pt,a4paper,reqno]{amsart}
\usepackage{amsmath}
\usepackage{amsfonts}
\usepackage{mathrsfs}
\usepackage{amssymb}
\usepackage{amsthm}
\usepackage{enumerate}
\usepackage[all]{xy}
\usepackage{graphicx}
\usepackage{fullpage}
\usepackage[bookmarksnumbered=true]{hyperref}
\DeclareMathOperator{\Var}{Var} \DeclareMathOperator{\dist}{dist}
\DeclareMathOperator{\CR}{CR}

\theoremstyle{plain}

  \newtheorem{proposition}[]{Proposition}
  \newtheorem{lemma}[]{Lemma}
  \newtheorem{theorem}[]{Theorem}
  \newtheorem{corollary}[]{Corollary}
  
  \newtheorem{remark}[]{Remark}
  \newtheorem{question}[]{Question}

\theoremstyle{definition}

\title{Asymptotics for 2D critical and near-critical first-passage percolation}
\author{Chang-Long Yao}
\thanks{Academy of Mathematics and Systems Science,
CAS, Beijing, China (E-mail: deducemath@126.com)}

\begin{document}
\maketitle

\begin{abstract}
We study Bernoulli first-passage percolation (FPP) on the triangular
lattice in which sites have 0 and 1 passage times with probability
$p$ and $1-p$, respectively.  Denote by $\mathcal {C}_{\infty}$ the
infinite cluster with 0-time sites when $p>p_c$, where $p_c=1/2$ is
the critical probability.  Denote by $T(0,\mathcal {C}_{\infty})$
the passage time from the origin 0 to $\mathcal {C}_{\infty}$. First
we obtain explicit limit theorem for $T(0,\mathcal {C}_{\infty})$ as
$p\searrow p_c$.  The proof relies on the limit theorem in the
critical case, the critical exponent for correlation length and
Kesten's scaling relations.   Next, for the usual point-to-point
passage time $a_{0,n}$ in the critical case, we construct
subsequences of sites with different growth rate along the axis. The
main tool involves the large deviation estimates on the nesting of
CLE$_6$ loops derived by Miller, Watson and Wilson (2016). Finally,
we apply the limit theorem for critical Bernoulli FPP to a random
graph called cluster graph, obtaining explicit strong law of large
numbers for graph distance.

\textbf{Keywords}: percolation; first passage percolation;
correlation length; scaling limit; conformal loop ensemble

\textbf{AMS 2010 Subject Classification}: 60K35, 82B43
\end{abstract}

\section{Introduction}
\subsection{The model}
Standard first-passage percolation (FPP) is defined on the integer
lattice $\mathbb{Z}^d$, where i.i.d. non-negative random variables
are assigned to nearest-neighbor edges.  This setting is called the
bond version of FPP on $\mathbb{Z}^d$.  We refer the reader to the
recent survey \cite{ADH17}.  In this paper, we will focus on the
site version of FPP on the triangular lattice $\mathbb{T}$, since
our main results rely on the existence of the scaling limit of
critical site percolation on $\mathbb{T}$.

The model is defined as follows.  Let
$\mathbb{T}=(\mathbb{V},\mathbb{E})$ denote the triangular lattice
embedded in $\mathbb{C}$, where $\mathbb{V}=\{x+ye^{\pi
i/3}\in\mathbb{C}:x,y\in\mathbb{Z}\}$ is the set of sites, and
$\mathbb{E}$ is the set of bonds, connecting adjacent sites.  Let
$\{t(v):v\in\mathbb{V}\}$ be an i.i.d. family of nonnegative random
variables with common distribution function $F$.   A \textbf{path}
is a sequence $(v_0,\ldots,v_n)$ of distinct sites such that
$v_{i-1}$ and $v_i$ are neighbors for all $i=1,\ldots,n$.  For a
path $\gamma$, we define its passage time as $T(\gamma)=\sum_{v\in
\gamma}t(v).$  The \textbf{first-passage time} between two site sets
$A,B$ is defined as
\begin{equation*}
T(A,B):=\inf \{T(\gamma):\gamma \mbox{ is a path from a site in $A$
to a site in $B$}\}.
\end{equation*}
A \textbf{geodesic} is a path $\gamma$ from $A$ to $B$ such that
$T(\gamma)= T(A,B)$.

Define the point-to-point passage time $a_{0,n}:=T(\{0\},\{n\})$. It
is well known that, based on subadditive ergodic theorem, if
$\mathbf{E}[t(v)]<\infty$, there is a constant $\mu=\mu(F)$ called
the \textbf{time constant}, such that
\begin{equation*}
\lim_{n\rightarrow\infty}\frac{a_{0,n}}{n}=\mu\quad\mbox{a.s. and in
$L^1$}.
\end{equation*}
Kesten (Theorem 6.1 in \cite{Kes86}) showed that
\begin{equation*}
\mu=0\quad\mbox{if and only if}\quad F(0)\geq p_c,
\end{equation*}
where $p_c=1/2$ is the critical probability for Bernoulli site
percolation on $\mathbb{T}$ (see e.g. \cite{Gri99} for general
background on percolation).  So one gets little information from the
time constant when $F(0)\geq p_c$.  When $F(0)=p_c$, we call the
model \textbf{critical FPP} since there is a transition of the time
constant at $p_c$.

In this paper, we shall restrict ourselves to \textbf{Bernoulli FPP}
on $\mathbb{T}$:  For each $p\in [0,1]$, we define the measure
$\mathbf{P}_p$ as the one under which all coordinate functions
$\{t(v):v\in \mathbb{V}\}$ are i.i.d. with
\begin{equation*}
\mathbf{P}_p[t(v)=0]=p=1-\mathbf{P}_p[t(v)=1],
\end{equation*}
and refer to a site $v$ with $t(v)=0$ simply as \textbf{open};
otherwise, \textbf{closed}.  One can view our Bernoulli FPP as
Bernoulli site percolation on $\mathbb{T}$.  We usually represent it
as a random coloring of the faces of the dual hexagonal lattice
$\mathbb{H}$, each face centered at $v\in \mathbb{V}$ being blue
($t(v)=0$) or yellow ($t(v)=1$).  Sometimes we view the site $v$ as
the hexagon in $\mathbb{H}$ centered at $v$.  Denote by $c_n$ the
first-passage time from 0 to a circle of radius $n$ centered at 0.
See Figure \ref{fig1}.  Using conformal loop ensemble CLE$_6$, the
author \cite{Yao18} gave the following limit theorem in the critical
case.
\begin{figure}
\begin{center}
\includegraphics[height=0.35\textwidth]{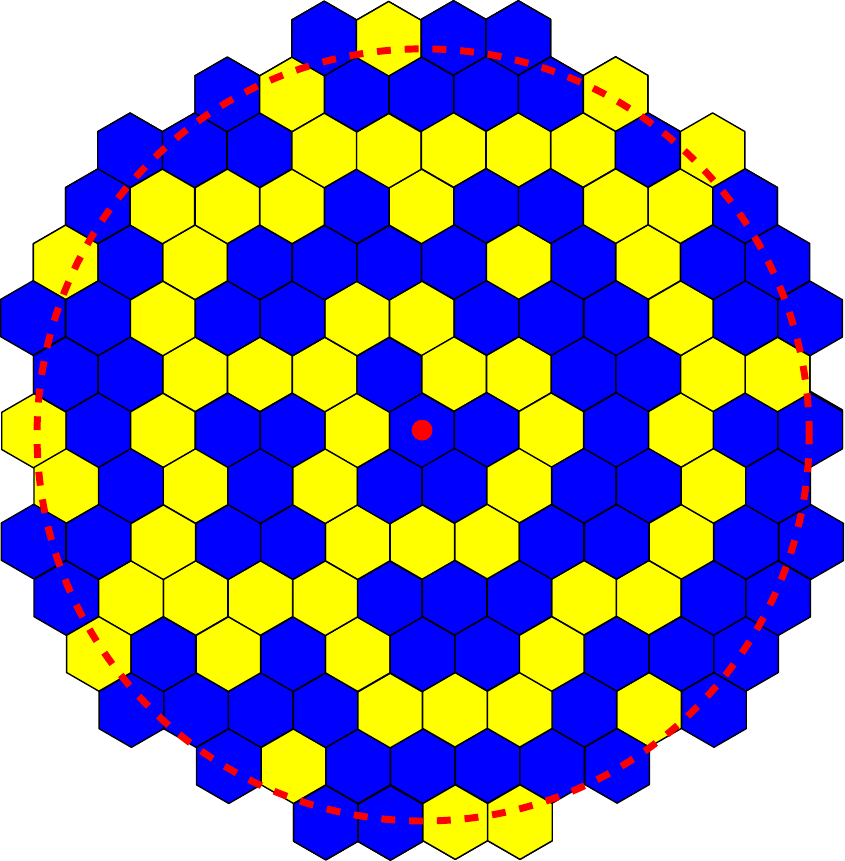}
\caption{Bernoulli site percolation on the triangular lattice
$\mathbb{T}$.  Each hexagon of the hexagonal lattice $\mathbb{H}$
represents a site of $\mathbb{T}$, and is colored blue ($t(v)=0$) or
yellow ($t(v)=1$). Here, the first-passage time
$c_6=2$.}\label{fig1}
\end{center}
\end{figure}

\begin{theorem}[\cite{Yao18}]\label{t4}
Under the critical Bernoulli FPP measure $\mathbf{P}_{1/2}$ on
$\mathbb{T}$,
\begin{align}
&\lim_{n\rightarrow\infty}\frac{c_n}{\log
n}=\frac{1}{2\sqrt{3}\pi}\quad a.s.,\label{t41}\\
&\lim_{n\rightarrow\infty}\frac{\mathbf{E}_{1/2}[c_n]}{\log
n}=\frac{1}{2\sqrt{3}\pi},\label{t42}\\
&\lim_{n\rightarrow\infty}\frac{\Var_{1/2}[c_n]}{\log
n}=\frac{2}{3\sqrt{3}\pi}-\frac{1}{2\pi^2},\label{t43}\\
&\lim_{n\rightarrow\infty}\frac{a_{0,n}}{\log
n}=\frac{1}{\sqrt{3}\pi}\quad\mbox{in probability but not
a.s.}\label{t44}
\end{align}
\end{theorem}
Let us mention that an analogous theorem for critical Bernoulli FPP
starting on the boundary was established in \cite{JY16}.  In
\cite{DLW17,KZ97}, the authors studied asymptotics for general
planar critical FPP.  From Theorem 1.6 in \cite{DLW17} (see also
(1.13) in \cite{KZ97}), we know that
\begin{equation*}
\frac{c_n-\mathbf{E}_{1/2}[c_n]}{\sqrt{\Var_{1/2}[c_n]}}\stackrel{d}\longrightarrow
N(0,1)\quad\mbox{as }n\rightarrow\infty.
\end{equation*}
This combined with (\ref{t42}) and (\ref{t43}) implies that there
exists a function $\delta(n)$ with $\delta(n)\rightarrow 0$ as
$n\rightarrow \infty$, such that
\begin{equation}\label{t46}
\frac{c_n-(1+\delta(n))\log
n/(2\sqrt{3}\pi)}{\sqrt{\left(2/(3\sqrt{3}\pi)-1/(2\pi^2)\right)\log
n}}\stackrel{d}\longrightarrow N(0,1)\quad\mbox{as
}n\rightarrow\infty.
\end{equation}
We conjecture, but can not prove, that one may choose $\delta\equiv
0$ in (\ref{t46}).  Let us point out that the explicit form of the
CLT in Corollary 1.2 of \cite{Yao18} should be replaced with a
similar weaker form.

In the present paper, we continue our study of Bernoulli FPP on
$\mathbb{T}$ from \cite{Yao18}.  The main purpose is threefold: to
derive exact asymptotics for Bernoulli FPP as $p\searrow p_c$, to
construct different subsequential limits for $a_{0,n}$ when $p=p_c$,
to obtain strong law of large numbers for a natural random graph by
applying the result for critical Bernoulli FPP.  See the next
subsection for details.

Throughout this paper, $C$ or $C_i$ stands for a positive constant
that may change from line to line according to the context.

\subsection{Main results}
Before stating our main results, we give some basic notation.  For
$r>0$, let $\mathbb{D}(r)$ denote the Euclidean disc of radius $r$
centered at 0 and $\partial\mathbb{D}(r)$ denote the boundary of
$\mathbb{D}(r)$.  Write $\mathbb{D}:=\mathbb{D}(1)$.  For $v\in
\mathbb{V}$, let $B(v,r)$ denote the set of hexagons of
$\mathbb{\mathbb{H}}$ that are contained in $v+\mathbb{D}(r)$.  We
will sometimes see $B(v,r)$ as a union of these closed hexagons. For
$B(v,r)$, denote by $\partial B(v,r)$ its topological boundary.
Write $B(r):=B(0,r)$ and $T(0,\partial
B(r)):=T(0,\mathbb{C}\backslash B(r))$.

Recall the standard coupling of the percolation measures
$\mathbf{P}_p,0\leq p\leq 1$: Take i.i.d. random variables $U_v$ for
each site $v$ of $\mathbb{V}$, with $U_v$ uniformly distributed on
$[0,1]$.  We denote the underlying probability measure by
$\mathbf{P}$, the corresponding expectation by $\mathbf{E}$, and the
space of configurations by $([0,1]^{\mathbb{V}},\mathscr{F})$, where
$\mathscr{F}$ is the cylinder $\sigma$-field on
$[0,1]^{\mathbb{V}}$.  For each $p$, we obtain the measure
$\mathbf{P}_p$ by declaring each site $v$ to be $p$-open ($t(v)=0$)
if $U_v\leq p$, and $p$-closed ($t(v)=1$) otherwise.  Let
$\mathbf{E}_p$ denote the expectation with respect to
$\mathbf{P}_p$.  It is well known that almost surely for each
$p>1/2$, there is a unique infinite open cluster, denoted by
$\mathcal {C}_{\infty}=\mathcal {C}_{\infty}(p)$.  Under the
coupling measure $\mathbf{P}$, denote by $T_p(\cdot,\cdot)$ the
first-passage time between two site sets with respect to the
parameter $p$.

Let $L(p)$ denote the correlation length that will be defined in
Section \ref{pre}.

We use the standard notation to express that two quantities are
asymptotically equivalent.  Given two positive functions $f$ and
$g$, we write $f\asymp g$ if there are constants $0<C_1<C_2<\infty$
such that $C_1g\leq f\leq C_2g$.

\subsubsection{Near-critical behavior: supercritical phase} The
following theorem roughly says that $T_p(0,\mathcal
{C}_{\infty}(p))$ is well-approximated by $T_{1/2}(0,\partial
B(L(p)))$ under the coupling measure $\mathbf{P}$.
\begin{theorem}\label{t3}
There exists some absolute constant $C>0$, such that for all
$p>1/2$,
\begin{align}
&\mathbf{E}|T_p(0,\mathcal {C}_{\infty}(p))-T_{1/2}(0,\partial
B(L(p)))|\leq
C,\label{t31}\\
&|\Var_p[T(0,\mathcal {C}_{\infty})]-\Var_{1/2}[T(0,\partial
B(L(p)))]|\leq C.\label{t35}
\end{align}
\end{theorem}
Assume $p>1/2$, it is easy to show that
\begin{equation*}
\lim_{n\rightarrow \infty}c_n=T(0,\mathcal
{C}_{\infty})\quad\mbox{$\mathbf{P}_p$-a.s.},~~\lim_{n\rightarrow
\infty}\mathbf{E}_p[c_n]=\mathbf{E}_p[T(0,\mathcal
{C}_{\infty})],~~\lim_{n\rightarrow
\infty}\mathbf{E}_p[a_{0,n}]=2\mathbf{E}_p[T(0,\mathcal
{C}_{\infty})].
\end{equation*}
Zhang \cite{Zha95} proved analogous result for general supercritical
FPP on $\mathbb{Z}^d$ with $F(0)>p_c$.  Using Theorem \ref{t4} and
Theorem \ref{t3}, we obtain exact asymptotics for $T_p(0,\mathcal
{C}_{\infty}(p))$ as $p\searrow 1/2$.
\begin{corollary}\label{c1}
Suppose $p>1/2$.  We have
\begin{align}
&\lim_{p\searrow 1/2}\frac{T_p(0,\mathcal
{C}_{\infty}(p))}{-\frac{4}{3}\log (p-1/2)}=
\frac{1}{2\sqrt{3}\pi}\quad\mbox{$\mathbf{P}$-a.s.,}\label{t32}\\
&\lim_{p\searrow 1/2}\frac{\mathbf{E}_p[T(0,\mathcal
{C}_{\infty})]}{-\frac{4}{3}\log
(p-1/2)}=\frac{1}{2\sqrt{3}\pi},\label{t33}\\
&\lim_{p\searrow 1/2}\frac{\Var_p[T(0,\mathcal
{C}_{\infty})]}{-\frac{4}{3}\log
(p-1/2)}=\frac{2}{3\sqrt{3}\pi}-\frac{1}{2\pi^2}.\label{t34}
\end{align}
Furthermore, there exists a function $\eta(p)$ with
$\eta(p)\rightarrow 0$ as $p\searrow 1/2$, such that
\begin{equation}\label{t36}
\frac{T_p(0,\mathcal
{C}_{\infty}(p))+(1+\eta(p))\frac{2}{3\sqrt{3}\pi}\log(p-1/2)}{\sqrt{\left(\frac{1}{2\pi^2}-\frac{2}{3\sqrt{3}\pi}\right)\frac{4}{3}\log
(p-1/2)}}\stackrel{d}\longrightarrow N(0,1)\quad\mbox{as $p\searrow
1/2$}.
\end{equation}
\end{corollary}
\begin{remark}
It is natural to ask what will happen when $p\nearrow 1/2$ for the
subcritical Bernoulli FPP.  Denote by $\mu(p)$ the corresponding
time constant.  Chayes, Chayes and Durrett \cite{CCD86} proved that
$\mu(p)\asymp 1/L(p)$. This result together with (\ref{e14}) implies
that $\mu(p)=(1/2-p)^{4/3+o(1)}$ as $p\nearrow 1/2$.  Denote by
$\mathcal {B}(p)$ the limit shape in the classical ``shape theorem"
(see e.g. Section 2 in \cite{ADH17}).  It will be proved in
\cite{Yao18+} that $(1/L(p))\mathcal {B}(p)$ converges to a
Euclidean disk as $p\nearrow 1/2$.  The proof relies on the scaling
limit of near-critical percolation constructed by Garban, Pete and
Schramm \cite{GPS18}.
\end{remark}

\subsubsection{Subsequential limits for critical Bernoulli
FPP}\label{subseqmain} Notice that in (\ref{t44}), we have
convergence of $a_{0,n}/\log n$ only in probability, not almost
surely.  In fact, to show that the convergence does not occur almost
surely, in the proof of Theorem 1.1 in \cite{Yao14}, we constructed
a random subsequence that converges to one half of the typical
limiting value almost surely. Using large deviation estimates on the
nesting of CLE$_6$ loops derived by Miller, Watson and Wilson
\cite{MWW16}, we get the following theorem. Loosely speaking, it
says that one can find many subsequences of sites with different
growth rate, growing unusually quickly or slowly.
\begin{theorem}\label{t1}
Let $\nu_1$ be the constant defined in Section \ref{subseq}.
$\mathbf{P}_{1/2}$-almost surely, for each $\nu\in[0,\nu_1]$, there
exists a random subsequence $\{n_i:i\geq 1\}$ depending on $\nu$,
such that
\begin{equation}\label{t12}
\lim_{i\rightarrow\infty}\frac{a_{0,n_i}}{\log
n_i}=\frac{1}{2\sqrt{3}\pi}+\nu.
\end{equation}
\end{theorem}

Let us mention that $\nu_1>1/(2\sqrt{3}\pi)$; one can derive more
accurate approximation of $\nu_1$ from its definition.  Note that
(\ref{t41}) implies that $\liminf_{n\rightarrow\infty}a_{0,n}/\log
n\geq 1/(2\sqrt{3}\pi)$ a.s.  For the $\limsup$ we propose the
following question.
\begin{question}\label{q1}
Show that under the measure $\mathbf{P}_{1/2}$,
\begin{equation*}
\limsup_{n\rightarrow\infty}\frac{a_{0,n}}{\log
n}=\frac{1}{2\sqrt{3}\pi}+\nu_1\quad\mbox{a.s.}
\end{equation*}
\end{question}
Question \ref{q2} below is a discrete analog of Lemma 3.2 of
\cite{MWW16}, and is related to Question \ref{q1}.
\begin{question}\label{q2}
Suppose $\nu\geq 0$.  Show that for all functions $\delta(n)$
decreasing to 0 sufficiently slowly as $n\rightarrow \infty$, we
have
\begin{equation*}
\left\{
\begin{aligned}
&\mathbf{P}_{1/2}\left[\nu\leq \frac{c_n}{\log n}\leq
\nu+\delta(n)\right]=n^{-\gamma(\nu)+o(1)}, &\mbox{for }\nu>0,\\
&\mathbf{P}_{1/2}\left[\frac{\delta(n)}{2}\leq \frac{c_n}{\log
n}\leq \delta(n)\right]=n^{-5/48+o(1)}, &\mbox{for }\nu=0.
\end{aligned}
\right.
\end{equation*}
\end{question}
\begin{remark}
Lemma \ref{l20} implies that the left hand side of the above
equation is not smaller than $n^{-\gamma(\nu)+o(1)}$. The remaining
task is to bound it in the other direction.

Note that $5/48$ in the above equation is the value of 1-arm
exponent (see e.g. Theorem 21 of \cite{Nol08}).  It is clear that
\begin{equation*}
\mathbf{P}_{1/2}[c_n=0]=(1/2)\mathbf{P}_{1/2}[\mathcal
{A}_1(1,n)]=n^{-5/48+o(1)}.
\end{equation*}
(See Section \ref{pre} for the definition of 1-arm event $\mathcal
{A}_1(1,n)$.) Proposition 18 of \cite{Nol08} concerns arms with
``defects" (i.e. sites of the opposite color), and implies that for
each fixed number $d\in \mathbb{N}$,
\begin{equation*}
\mathbf{P}_{1/2}[c_n\leq d]\asymp (1+\log n)^d
\mathbf{P}_{1/2}[\mathcal {A}_1(1,n)]=n^{-5/48+o(1)}.
\end{equation*}
\end{remark}

\subsubsection{Application to cluster graph}

In this section, we shall introduce a model called cluster graph. It
is a natural object constructed from critical percolation.  Then we
give an application of the limit theorem for critical Bernoulli FPP
to this model.

\begin{figure}
\begin{center}
\includegraphics[height=0.3\textwidth]{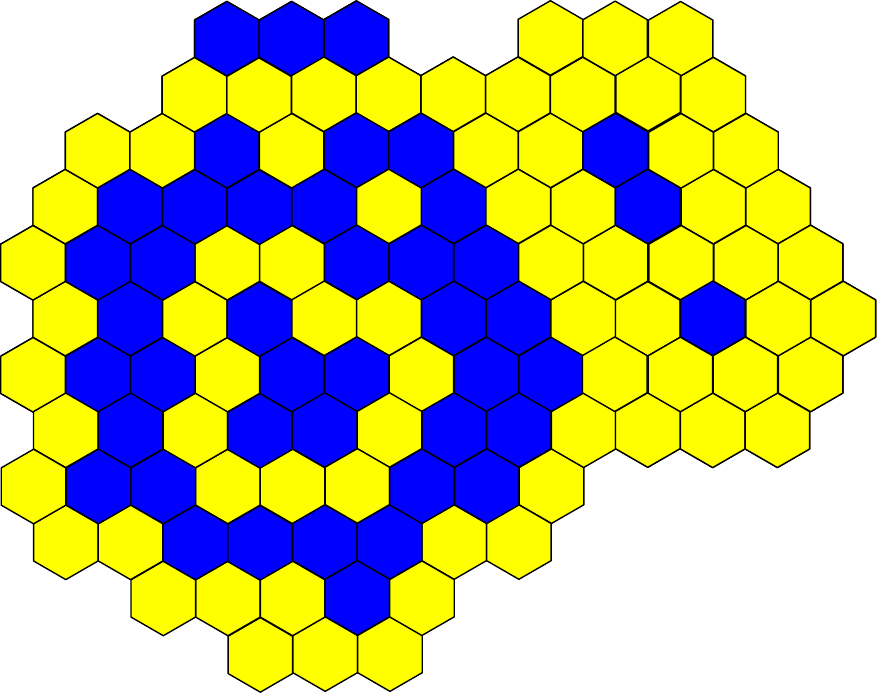}
\caption{The three open (blue) clusters on the left belong to the
same component of cluster graph.  The two open clusters on the right
are neighbors and belong to a finite component containing only the
two clusters.}\label{fig2}
\end{center}
\end{figure}

\emph{Cluster graph}.  Benjamini \cite{Ben13} studied some random
metric spaces modeled by graphs.  Based on the bond percolation on
$\mathbb{Z}^d$, in Section 10.2 of \cite{Ben13} he defined a random
graph called Contracting Clusters of Critical Percolation (CCCP) by
the following rule: Contract each open cluster into a single vertex
and define a new edge between the clusters $\mathcal {C},\mathcal
{C}'$ for every closed edge that connects them in $\mathbb{Z}^d$.
Similarly, let us define \textbf{cluster graph} based on critical
site percolation on $\mathbb{T}$: Contract each open cluster into a
single vertex and define a new edge between any pair of clusters
$\mathcal {C},\mathcal {C}'$ if there exists a closed hexagon
touching both of $\mathcal {C}$ and $\mathcal {C}'$.  Unlike
Benjamini's CCCP which is almost surely a connected multi-graph, our
cluster graph is a simple graph and has infinitely many components
almost surely.  See Figure \ref{fig2}.  We embed the cluster graph
in the plane in a natural way: Each open cluster is viewed as a
vertex of the cluster graph.  One may imagine open clusters as
islands and closed clusters as lakes, so one cannot cross the water
of width larger than the diameter of one hexagon.
\begin{proposition}\label{p5}
Cluster graph has a unique infinite component almost surely, denoted
by $\mathscr{C}$.  There exists a constant $C>0$, such that for any
$k\geq 1$,
\begin{equation}\label{e91}
\mathbf{P}_{1/2}[\dist(0,\mathscr{C})\geq k]\leq \exp(-Ck),
\end{equation}
where $\dist(\cdot,\cdot)$ denotes Euclidean distance.
\end{proposition}

To state the following theorem, we need some notation.  Let
$D(\cdot,\cdot)$ denote the graph distance in the cluster graph. For
$n\in\mathbb{N}$, denote by $\mathcal {C}_n$ the innermost open
cluster surrounding $B(n)$.  Let $\mathcal {C}_0$ be the open
cluster containing 0.  Note that if 0 is closed, $\mathcal
{C}_0=\emptyset$. Using (\ref{t41}), we derive the following strong
law of large numbers for the graph distance in cluster graph.
\begin{theorem}\label{t6}
Under the conditional probability measure
$\mathbf{P}_{1/2}[\cdot|\mathcal {C}_0\in \mathscr{C}]$,
\begin{equation*}
\lim_{n\rightarrow\infty}\frac{D(\mathcal {C}_0,\mathcal
{C}_n)}{\log n}=\frac{1}{2\sqrt{3}\pi}\quad\mbox{a.s.}
\end{equation*}
\end{theorem}
It is worth mentioning another application of critical Bernoulli FPP
to loop graph defined below.  We will just state the result without
giving the proof.  We note that the proof is very similar to that
for cluster graph.

\emph{Loop graph}.  It is well known that Camia and Newman
\cite{CN06} proved that the scaling limit of cluster boundaries of
critical site percolation on $\mathbb{T}$ is CLE$_6$.  Several
properties of this scaling limit are established; the third item in
Theorem 2 of \cite{CN06} was called ``finite chaining" property in
Proposition 2.7 in \cite{GPS13}.  That is, for the full-plane
CLE$_6$, almost surely any two loops are connected by a finite
¡°path¡± of touching loops.  It is natural to define a discrete
version of this notion.  Similarly to cluster graph, for critical
site percolation on $\mathbb{T}$, we contract each cluster boundary
loop into a single vertex and define a new edge between any pair of
loops $\mathcal {L},\mathcal {L}'$ if there exists a hexagon
touching both of $\mathcal {L}$ and $\mathcal {L}'$.  Then we get a
random graph called \textbf{loop graph}, whose vertices correspond
to cluster boundary loops.

Let $D_l(\cdot,\cdot)$ denote the graph distance in the loop graph.
For $n\in \mathbb{N}$, denote by $\mathcal {L}_n$ the innermost
cluster boundary loop surrounding $\mathbb{D}(n)$. Let $\mathcal
{L}_0$ be the innermost cluster boundary loop surrounding 0.

Similarly to cluster graph, loop graph also has a unique infinite
component almost surely, denoted by $\mathscr{C}_l$.  Moreover,
under $\mathbf{P}_{1/2}[\cdot|\mathcal {L}_0\in \mathscr{C}_l]$,
\begin{equation*}
\lim_{n\rightarrow\infty}\frac{D_l(\mathcal {L}_0,\mathcal
{L}_n)}{\log n}=\frac{1}{\sqrt{3}\pi}\quad\mbox{a.s.}
\end{equation*}

\section{Notation and preliminaries}\label{pre}
Our proofs rely heavily on critical and near-critical percolation.
In this section we collect some results that are needed.

A \textbf{circuit} is a path $(v_1,\ldots,v_n)$ with $n\geq 3$, such
that $v_1$ and $v_n$ are neighbors.  Note that the bonds
$(v_1,v_2),\ldots,(v_n,v_1)$ of the circuit form a Jordan curve, and
sometimes the circuit is viewed as this curve.  For a circuit
$\mathcal {C}$, define
\begin{equation*}
\overline{\mathcal {C}}:= \mathcal {C}\cup\mbox{ interior sites of
}\mathcal {C}.
\end{equation*}

If $W$ is a set of sites, then its (external) site boundary is
\begin{equation*}
\Delta W:=\{v:v\notin W\mbox{ but $v$ is adjacent to $W$}\}.
\end{equation*}

Given $\varepsilon\in(0,1)$ and $p\in (1/2,1]$, we define the
\textbf{correlation length} (or characteristic length) by
\begin{equation*}
L_{\varepsilon}(p):=\min\{n\in\mathbb{N}: \mathbf{P}_p[\mbox{there
is an open horizontal crossing of $[0,n]^2$}]\geq 1-\varepsilon\},
\end{equation*}
and by $L_{\varepsilon}(p):=L_{\varepsilon}(1-p)$ for $p<1/2$.  We
will take the convention $L_{\varepsilon}(1/2)=\infty$.

For two positive functions $f$ and $g$ from a set $\mathcal {X}$ to
$(0,\infty)$, we write $f(x)\asymp g(x)$ to indicate that
$f(x)/g(x)$ is bounded away from 0 and $\infty$, uniformly in $x\in
\mathcal {X}$. It is well known (see e.g. \cite{Nol08}) that there
exists $\varepsilon_0>0$ such that for all
$0<\varepsilon_1,\varepsilon_2\leq \varepsilon_0$ we have
$L_{\varepsilon_1}(p)\asymp L_{\varepsilon_2}(p)$.  For simplicity
we will write $L(p):=L_{\varepsilon_0}(p)$ for the entire paper.

For $1\leq r\leq R$ and $v\in \mathbb{V}$, define annuli
\begin{align*}
A(v;r,R):=B(v,R)\backslash B(v,r),~~A(r,R):=A(0;r,R).
\end{align*}

The so-called arm events play a central role in studying
near-critical percolation.  A color sequence $\sigma$ is a sequence
$(\sigma_1,\sigma_2,\dots,\sigma_k)$ of ``blue" and ``yellow" of
length $k$.  We write 0 and 1 for blue and yellow, respectively.  We
identify two sequences if they are the same up to a cyclic
permutation.  For an annulus $A(v;r,R)$, we denote by $\mathcal
{A}_{\sigma}(v;r,R)$ the event that there exist $|\sigma|=k$
disjoint monochromatic paths called \textbf{arms} in $A(v;r,R)$
connecting the two boundary pieces of $A(v;r,R)$, whose colors are
those prescribed by $\sigma$, when taken in counterclockwise order.

For simplicity, for any $r\geq 1$, we let $\mathcal
{A}_{\sigma}(v;r,r)$ be the entire sample space $\Omega$.  Write
$\mathcal {A}_{\sigma}(r,R)=\mathcal {A}_{\sigma}(0;r,R)$ and
$\mathcal {A}_1=\mathcal {A}_{(0)}$, $\mathcal {A}_2=\mathcal
{A}_{(01)}$, $\mathcal {A}_4=\mathcal {A}_{(0101)}$, $\mathcal
{A}_6=\mathcal {A}_{(011011)}$.

Let $\mathcal {O}(r,R)$ denote the event that there exists a blue
circuit surrounding 0 in $A(r,R)$.  Note that $\mathcal
{O}(r,R)=\mathcal {A}_{(1)}^c(r,R)$.

We assume that the reader is familiar with the FKG inequality (see
Lemma 13 in \cite{Nol08} for generalized FKG), the BK (van den
Berg-Kesten) inequality, and the RSW (Russo-Seymour-Welsh)
technology.  Here we collect some classical results in near-critical
percolation that will be used. See e.g. \cite{Nol08,Wer09} and
Section 2.2 in \cite{BKN18}.

\begin{enumerate}[(i)]
\item \emph{A priori bounds for arm events}:  For any color sequence $\sigma$, there
exist $C_1(|\sigma|)$, $C_2(|\sigma|)$, $\alpha(|\sigma|)$,
$\beta(|\sigma|)>0$ such that for all $1\leq r<R$,
\begin{equation*}
C_1\left(\frac{r}{R}\right)^{\alpha}\leq \mathbf{P}_{1/2}[\mathcal
{A}_{\sigma}(r,R)]\leq C_2\left(\frac{r}{R}\right)^{\beta}.
\end{equation*}
\item \emph{Extendability}: For any color sequence $\sigma$,
\begin{equation*}
\mathbf{P}_p[\mathcal {A}_{\sigma}(r,2R)]\asymp
\mathbf{P}_p[\mathcal {A}_{\sigma}(r,R)]
\end{equation*}
uniformly in $p$ and $1\leq r\leq R\leq L(p)$.
\item \emph{Quasi-multiplicativity}: For any color sequence $\sigma$, there
exists $C(|\sigma|)>0$, such that
\begin{equation*}
C\mathbf{P}_p[\mathcal {A}_{\sigma}(r_1,r_2)]\mathbf{P}_p[\mathcal
{A}_{\sigma}(r_2,r_3)]\leq \mathbf{P}_p[\mathcal
{A}_{\sigma}(r_1,r_3)]
\end{equation*}
uniformly in $p$ and $1\leq r_1<r_2<r_3\leq L(p)$.
\item For any color sequence $\sigma$,
\begin{equation}\label{e7}
\mathbf{P}_p[\mathcal {A}_{\sigma}(r,R)]\asymp
\mathbf{P}_{1/2}[\mathcal {A}_{\sigma}(r,R)]
\end{equation}
uniformly in $p$ and $1\leq r<R\leq L(p)$.
\item As $p\rightarrow 1/2$,
\begin{equation}\label{e8}
|p-1/2|(L(p))^2\mathbf{P}_{1/2}[\mathcal {A}_4(1,L(p))]\asymp 1.
\end{equation}
\item  \emph{Exponential decay with respect to $L(p)$}.  There are constants $C_1,C_2>0$, such that for all
$p>1/2$ and $R\geq L(p)$ (see item (ii) in Section 2.2 of
\cite{BKN18}),
\begin{align}
&\mathbf{P}_p[\mathcal {A}_1(R,\infty)]\geq
1-C_1\exp\left(-C_2\frac{R}{L(p)}\right),\label{e11}\\
&\mathbf{P}_p[\mathcal {O}(R,2R)]\geq
1-C_1\exp\left(-C_2\frac{R}{L(p)}\right).\label{e40}
\end{align}
\item There exist constants $\varepsilon,C>0$, such that for all $p$ and $1\leq r<R\leq L(p)$,
\begin{equation}\label{e9}
\mathbf{P}_p[\mathcal {A}_4(r,R)]\geq
C\left(\frac{r}{R}\right)^{2-\varepsilon}.
\end{equation}
\item There exist constants $\varepsilon,C>0$, such that for all $1\leq r<R$,
\begin{equation}\label{e67}
\mathbf{P}_{1/2}[\mathcal {A}_6(r,R)]\leq
C\left(\frac{r}{R}\right)^{2+\varepsilon}.
\end{equation}
\item When $p\rightarrow 1/2$,
\begin{equation}\label{e14}
L(p)=|p-1/2|^{-4/3+o(1)}.
\end{equation}
\end{enumerate}

It is well known that for a fixed number $j$ of arms, if its color
sequence is polychromatic (both colors are present), prescribing it
changes the probability only by at most a constant factor.  Beffara
and Nolin \cite{BN11} showed that the monochromatic $j$-arm exponent
is strictly between the polychromatic $j$-arm and $(j+1)$-arm
exponents.  The following result was essentially proved, see (2.10)
and the inequality just below (3.1) in \cite{BN11}.
\begin{lemma}[\cite{BN11}]\label{l4}
For any polychromatic color sequence $\sigma$, there exist
$\varepsilon,C>0$ (depending on $|\sigma|$), such that for all
$1\leq m<n$,
\begin{equation*}
\mathbf{P}_{1/2}[\mathcal {A}_{(\underbrace{1\ldots
1}_{|\sigma|})}(m,n)]\leq
C\left(\frac{n}{m}\right)^{-\varepsilon}\mathbf{P}_{1/2}[\mathcal
{A}_{\sigma}(m,n)].
\end{equation*}
\end{lemma}

We call a continuous map from the circle to $\mathbb{C}$ a
\textbf{loop}; the loops are identified up to reparametrization by
homeomorphisms of the circle with positive winding.  Let
$\textrm{d}(\cdot,\cdot)$ denote the uniform metric on loops:
\begin{equation*}
\textrm{d}(\gamma_1,\gamma_2):=\inf_{\phi}\sup_{t\in
\mathbb{R}/\mathbb{Z}}|\gamma_1(t)-\gamma_2(\phi(t))|,
\end{equation*}
where the infimum is taken over all homeomorphisms of the circle
which have positive winding.  The distance between two closed sets
of loops is defined by the induced Hausdorff metric as follows:
\begin{equation}\label{e66}
\dist(\mathcal {F},\mathcal {F}'):=\inf\{\varepsilon>0:\forall
\gamma\in \mathcal {F},\exists \gamma'\in\mathcal {F}'\mbox{ such
that }\textrm{d}(\gamma,\gamma')\leq\varepsilon\mbox{ and vice
versa}\}.
\end{equation}

For critical site percolation on $\mathbb{T}$, we orient a cluster
boundary loop counterclockwise if it has blue sites on its inner
boundary and yellow sites on its outer boundary, otherwise we orient
it clockwise.  We say $B(R)$ has \textbf{monochromatic (blue)
boundary condition} if all the sites in $\Delta B(R)$ are blue.
Based on Smirnov's celebrated work \cite{Smi01}, Camia and Newman
\cite{CN06} showed the following well-known result.

\begin{theorem}[\cite{CN06}]\label{t5}
As $\eta\rightarrow 0$, the collection of all cluster boundaries of
critical site percolation on $\eta\mathbb{T}$ in $\mathbb{D}$ with
monochromatic boundary conditions converges in distribution, under
the topology induced by metric (\ref{e66}), to a probability
distribution on collections of continuous nonsimple loops in
$\overline{\mathbb{D}}$.
\end{theorem}

We call the continuum nonsimple loop process in Theorem \ref{t5} the
conformal loop ensemble CLE$_6$ in $\overline{\mathbb{D}}$.  General
CLE$_\kappa$ for $8/3<\kappa<8$ is the canonical conformally
invariant measure on countably infinite collections of noncrossing
loops in a simply connected planar domain, see \cite{She09,SW12}. We
denote by $\mathbb{P}$ the probability measure of CLE$_6$ in
$\overline{\mathbb{D}}$ and by $\mathbb{E}$ the expectation with
respect to $\mathbb{P}$.

Given an annulus $A(r,R)$,  define
\begin{align*}
\rho(r,R)&:=\mbox{the maximal number of disjoint yellow circuits
surrounding 0 in }A(r,R),\\
N(r,R)&:=\mbox{the number of cluster boundary loops surrounding 0 in
$A(r,R)$}.
\end{align*}
The following elementary proposition is crucial for enabling us to
use the scaling limit of critical site percolation on $\mathbb{T}$
to derive explicit limit theorem for our special FPP model.  Note
that item (i) implies item (ii) immediately.
\begin{proposition}[Proposition 2.4 in \cite{Yao18}]\label{p3}
Consider Bernoulli FPP on $\mathbb{T}$ with parameter $p$.  Suppose
$1\leq r<R$. Then we have:
\begin{itemize}
\item[(i)] $T(\partial B(r),\partial B(R))=\rho(r,R)$.
\item[(ii)] There exist $T(\partial B(r),\partial B(R))$ disjoint yellow circuits
surrounding 0 in $A(r,R)$, such that for any geodesic $\gamma$ from
$\partial B(r)$ to $\partial B(R)$ in $A(r,R)$, each closed site in
$\gamma$ is in one of these circuits, with different closed sites
lying in different circuits.
\item[(iii)] Assume that $p=1/2$ and
$B(R)$ has monochromatic boundary condition.  Then $T(\partial
B(r),\partial B(R))$ has the same distribution as $N(r,R)$.
\end{itemize}
\end{proposition}

For $1\leq r<R$, denote by $S(r,R)$ (resp. $S_r$) the maximal number
of disjoint yellow circuits surrounding 0 and intersecting $A(r,R)$
(resp. $\partial B(r)$).

\begin{lemma}\label{l10}
There exist constants $C_1,\ldots,C_4>0$ and $K>1$, such that for
all $1\leq r<R$ and $x\geq K\log_2(R/r)$,
\begin{align}
&\mathbf{P}_{1/2}[\rho(r,R)\geq x]\leq C_1\exp(-C_2x),\label{e80}\\
&\mathbf{P}_{1/2}[S(r,R)\geq x]\leq C_3\exp(-C_4x).\label{e79}
\end{align}
Hence, there exists a constant $C_5>0$, such that for all $1\leq
r\leq R/2$,
\begin{equation*}
\mathbf{E}_{1/2}[S^4(r,R)]\leq C_5\log^4(R/r).
\end{equation*}
\end{lemma}
\begin{proof}
Combining item (i) in Proposition \ref{p3} and Lemma 2.5 in
\cite{Yao18}, we get (\ref{e80}).

Using RSW, FKG and BK inequality, it is easy to prove that there
exists a constant $C_6>0$, such that for all $r\geq 1$ and $x\geq
1$,
\begin{equation*}
\mathbf{P}_{1/2}[S_r\geq x]\leq \exp(-C_6x).
\end{equation*}
Since $S(r,R)\leq \rho(r,R)+S_r+S_R$, the above inequality and
(\ref{e80}) imply (\ref{e79}) immediately.
\end{proof}

\section{Supercritical regime}
In this section, we will prove Theorem \ref{t3} and Corollary
\ref{c1}.  We first introduce Russo's formula for random variables
in Section \ref{Russo}.  This formula plays a central role in the
proof of Theorem \ref{t3}, since it allows us to study how the
expectation of a random variable varies when the percolation
parameter $p$ varies. Then we prove (\ref{t31}) and (\ref{t35}) in
Sections \ref{mean} and \ref{var}, respectively.  The proof of
Corollary \ref{c1} is given in Section \ref{cor}.

For convenience, in the proofs of this section we always assume
without loss of generality that $L(p)$ is large, say, $L(p)\geq 20$.
So we suppose that $p\leq p_0$ for some fixed $p_0\in(1/2,1)$.  It
is easy to see by (\ref{e11}) that $\mathbf{E}_p[T^2(0,\mathcal
{C}_{\infty})]$ is uniformly bounded for $p\in[p_0,1]$, which
implies that (\ref{t31}) and (\ref{t35}) hold for $p\in[p_0,1]$
immediately.
\subsection{Russo's formula}\label{Russo}
We begin with some notation.  Given a percolation configuration
$\omega=\{\omega(u)\}_{u\in \mathbb{V}}\in\Omega$ and a site
$v\in\mathbb{V}$, let
\begin{equation*}
\omega^v(u):=\left\{
\begin{aligned}
&\omega(u) &\mbox{ if }u\neq v,\\
&0 &\mbox{if }u=v.
\end{aligned}
\right. \quad\omega_v(u):=\left\{
\begin{aligned}
&\omega(u) &\mbox{ if }u\neq v,\\
&1 &\mbox{if }u=v.
\end{aligned}
\right.
\end{equation*}
(Note that in the percolation literature $\omega^v$ usually means
that we set $v$ to be 1; here we adopt the above setting since, for
our Bernoulli FPP, a site is open when it takes the value 0, while
in the percolation literature, a site is open usually means that the
site takes the value 1.)  For a random variable $X=X(\omega)$,
define the increment of $X$ at $v$ by
\begin{equation*}
\delta_vX(\omega):=X(\omega^v)-X(\omega_v).
\end{equation*}
\begin{lemma}[Russo's formula, see e.g. Theorem 2.32 in \cite{Gri99}]\label{l6}
Let $X$ be a random variable which is defined in terms of the states
of only finitely many sites of $\mathbb{T}$.  Then
\begin{equation*}
\frac{d}{dp}\mathbf{E}_p[X]=\sum_{v\in
\mathbb{V}}\mathbf{E}_p[\delta_v X].
\end{equation*}
\end{lemma}
\subsection{Study of the mean}\label{mean}
Suppose $p>1/2$.  For simplicity of notation, let
$T(p):=T(0,\partial B(L(p)))$.  To prove (\ref{t31}), we write
\begin{align}
&\mathbf{E}|T_{1/2}(0,\partial
B(L(p)))-T_p(0,\mathcal {C}_{\infty}(p))|\nonumber\\
&\leq\mathbf{E}|T_{1/2}(0,\partial B(L(p)))-T_p(0,\partial
B(L(p)))|+\mathbf{E}|T_p(0,\partial B(L(p)))-T_p(0,\mathcal
{C}_{\infty}(p))|\nonumber\\
&=\{\mathbf{E}_{1/2}[T(p)]-\mathbf{E}_p[T(p)]\}+\{\mathbf{E}_p[T(0,\mathcal
{C}_{\infty})]-\mathbf{E}_p[T(p)]\}.\label{e74}
\end{align}
We will bound the two terms on the right-hand side of (\ref{e74})
separately, starting with the first term.
\begin{lemma}\label{l7}
There is a constant $C>0$ such that for all $p>1/2$,
\begin{equation*}
\mathbf{E}_{1/2}[T(p)]-\mathbf{E}_p[T(p)]\leq C.
\end{equation*}
\end{lemma}
\begin{proof}
For each $v\in B(L(p))$, define the event
\begin{equation*}
\mathcal {E}_v:=\{\mbox{for $\omega_v$, $\exists$ a geodesic
$\gamma$ from 0 to $\partial B(L(p))$ in $B(L(p))$ with $v\in
\gamma$}\}.
\end{equation*}
By Lemma \ref{l6}, applying Russo's formula to $T(p)$ for
$\mathbf{E}_h$, where $1/2\leq h\leq p$, one obtains
\begin{equation}\label{e73}
-\frac{d}{dh}\mathbf{E}_h[T(p)]=-\sum_{v\in
B(L(p))}\mathbf{E}_h[\delta_v T(p)]=\sum_{v\in
B(L(p))}\mathbf{P}_h[\mathcal {E}_v].
\end{equation}
Now let us show that there is a universal constant $C_1>0$, such
that for $v\in A(2,L(p))$,
\begin{equation}\label{e68}
\mathbf{P}_h[\mathcal {E}_v]\leq C_1\mathbf{P}_{1/2}[\mathcal
{A}_4(1,|v|)].
\end{equation}
Take $K=K(v)=\lfloor\log_2|v|\rfloor$.  We start by analyzing the
case that $v$ is far from the boundary of $B(L(p))$, that is, $v\in
A(2,L(p)/2)$.  Define the event
\begin{equation*}
\mathcal {F}_v:=\{\exists~r\mbox{ such that $1\leq r\leq 2^K$ and
$\mathcal {A}_4(v;1,r)\cap\mathcal {A}_{(1111)}(v;r,2^K)$ occurs}\}.
\end{equation*}
Note that $\mathcal {A}_4(v;1,2^K)\subset\mathcal {F}_v$ since we
have set $\mathcal {A}_{(1111)}(v;2^K,2^K)=\Omega$.

\begin{figure}
\begin{center}
\includegraphics[height=0.35\textwidth]{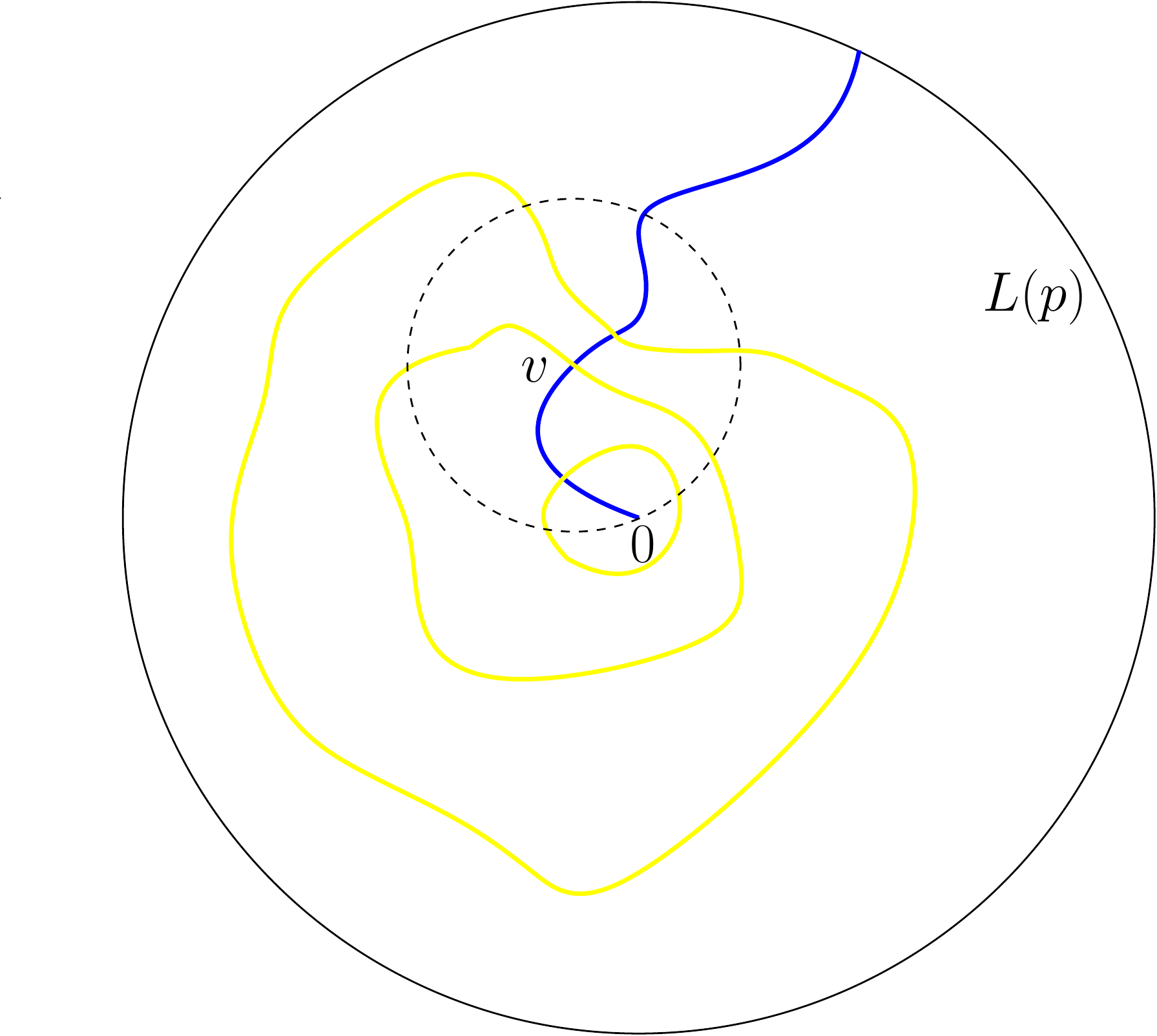}
\caption{For $v\in A(2,L(p)/2)$, the event $\mathcal
{E}_v\subset\mathcal {F}_v$.}\label{fig3}
\end{center}
\end{figure}

Assume that $v\in A(2,L(p)/2)$.  By item (ii) in Proposition
\ref{p3}, we have $\mathcal {E}_v\subset\mathcal {F}_v$.  See Figure
\ref{fig3}.  Let us bound the probability of the event $\mathcal
{F}_v$. By considering the smallest $r$ satisfying $\mathcal {F}_v$
with $2^i\leq r\leq 2^{i+1}$, we get that there exist universal
constants $\varepsilon,C_2,\dots,C_5>0$ such that
\begin{align}
\mathbf{P}_h[\mathcal {F}_v]&\leq
\sum_{i=0}^{K-1}\mathbf{P}_h[\mathcal
{A}_4(v;1,2^i)]\mathbf{P}_h[\mathcal {A}_{(1111)}(v;2^{i+1},2^K)]\nonumber\\
&\leq \sum_{i=0}^{K-1}C_2\mathbf{P}_{1/2}[\mathcal
{A}_4(1,2^i)]\mathbf{P}_{1/2}[\mathcal {A}_{(1111)}(2^{i+1},2^K)]~~\mbox{by (\ref{e7})}\nonumber\\
&\leq \sum_{i=0}^{K-1}C_3\mathbf{P}_{1/2}[\mathcal
{A}_4(1,2^i)]\mathbf{P}_{1/2}[\mathcal {A}_4(2^{i+1},2^K)]2^{-\varepsilon(K-i-1)}~~\mbox{by Lemma \ref{l4}}\nonumber\\
&\leq \sum_{i=0}^{K-1}C_4\mathbf{P}_{1/2}[\mathcal
{A}_4(1,2^K)]2^{-\varepsilon(K-i-1)}~~\mbox{ by
quasi-multiplicativity}\nonumber\\
&\leq C_5\mathbf{P}_{1/2}[\mathcal {A}_4(1,|v|)]~~\mbox{by
extendability.}\label{e70}
\end{align}
Then for $v\in A(2,L(p)/2)$, we get (\ref{e68}) from (\ref{e70})
since $\mathcal {E}_v\subset\mathcal {F}_v$.

Now we bound $\mathbf{P}_h[\mathcal {E}_v]$ for the sites $v$ which
are close to the boundary of $B(L(p))$, that is, $v\in
A(L(p)/2,L(p))$.  Let us mention that in the proof of Lemma
\ref{l7}, one can avoid analyzing this case by introducing an
intermediate measure $\widetilde{\mathbf{P}}_h$ satisfying
$\widetilde{\mathbf{P}}_h|_{A(L(p)/2,L(p))}=\mathbf{P}_{1/2}|_{A(L(p)/2,L(p))}$
and
$\widetilde{\mathbf{P}}_h|_{B(L(p)/2)}=\mathbf{P}_h|_{B(L(p)/2)}$.
However, in the study of the variance in Section \ref{var} we will
need to handle the boundary issue.  So we give the analysis here,
and will use it in Section \ref{var}.

Assume that $v\in A(L(p)/2,L(p))$ and $1<r\leq |v|$.  Define the
event
\begin{align*}
\widehat{\mathcal {A}}_4(v;1,r):=&\{\mbox{$\exists$ four alternating
arms from $v$ to $\partial (B(v,r)\cap B(L(p)))$ in $B(v,r)\cap B(L(p))$.}\\
&\mbox{Furthermore, three of them are from $v$ to $\partial B(v,r)$,
with color sequence $(101)$}\}.
\end{align*}
When $v$ touches $\partial B(L(p))$, we just interpret
$\widehat{\mathcal {A}}_4(v;1,r)$ as the event that there exist
three arms from $v$ to $\partial B(v,r)$ in $B(v,r)\cap B(L(p))$,
with color sequence $(101)$.  It is clear that when $r\leq
L(p)-|v|$, we have $\widehat{\mathcal {A}}_4(v;1,r)=\mathcal
{A}_4(v;1,r)$.

By using the fact that the polychromatic half-plane 3-arm exponent
is 2, which is larger than the 4-arm exponent, it is standard to
show that there is some universal constant $C_6>0$, such that
\begin{equation}\label{e71}
\mathbf{P}_h[\widehat{\mathcal {A}}_4(v;1,r)]\leq C_6
\mathbf{P}_{1/2}[\mathcal {A}_4(1,r)].
\end{equation}
Similarly to the event $\mathcal {F}_v$, for $v\in A(L(p)/2,L(p))$
we define the event
\begin{equation*}
\widehat{\mathcal {F}}_v:=\{\exists~r\mbox{ such that $1\leq r\leq
2^K$ and $\widehat{\mathcal {A}}_4(v;1,r)\cap\mathcal
{A}_{(1111)}(v;r,2^K)$ occurs}\}.
\end{equation*}
Suppose that $v\in A(L(p)/2,L(p))$.  Using item (ii) in Proposition
\ref{p3}, we have $\mathcal {E}_v\subset\widehat{\mathcal {F}}_v$.
By considering the smallest $r$ satisfying $\widehat{\mathcal
{F}}_v$ with $2^i\leq r\leq 2^{i+1}$,  we obtain that
\begin{align}
\mathbf{P}_h[\widehat{\mathcal {F}}_v]&\leq
\sum_{i=0}^{K-1}\mathbf{P}_h[\widehat{\mathcal
{A}}_4(v;1,2^i)]\mathbf{P}_h[\mathcal {A}_{(1111)}(v;2^{i+1},2^K)]\nonumber\\
&\leq \sum_{i=0}^{K-1}C_7\mathbf{P}_{1/2}[\mathcal
{A}_4(1,2^i)]\mathbf{P}_{1/2}[\mathcal {A}_4(2^{i+1},2^K)]2^{-\varepsilon(K-i-1)}~~\mbox{by (\ref{e71}), (\ref{e7}) and Lemma \ref{l4}}\nonumber\\
&\leq C_8\mathbf{P}_{1/2}[\mathcal {A}_4(1,|v|)]~~\mbox{by the proof
of (\ref{e70}).}\label{e72}
\end{align}
Then for $v\in A(L(p)/2,L(p))$, we derive (\ref{e68}) from
(\ref{e72}) since $\mathcal {E}_v\subset\widehat{\mathcal {F}}_v$.
This combined with the above argument for $v\in A(2,L(p)/2)$ ends
the proof of (\ref{e68}).

Take $M$ such that $2^M\leq L(p)<2^{M+1}$.  We have
\begin{align}
\sum_{v\in A(2,L(p))}\mathbf{P}_h[\mathcal {E}_v]&\leq
\sum_{i=1}^{M}\sum_{v\in
A(2^i,2^{i+1})}C_1\mathbf{P}_{1/2}[\mathcal {A}_4(1,|v|)]~~\mbox{ by (\ref{e68})}\nonumber\\
&\leq
\sum_{i=1}^{M}C_94^i\mathbf{P}_{1/2}[\mathcal {A}_4(1,2^i)]\nonumber\\
&\leq \sum_{i=1}^{M}C_{10}4^i\frac{\mathbf{P}_{1/2}[\mathcal
{A}_4(1,L(p))]}{\mathbf{P}_{1/2}[\mathcal {A}_4(2^i,L(p))]}~~\mbox{
by quasi-multiplicativity}\nonumber\\
&\leq
\sum_{i=1}^{M}C_{11}2^{-\varepsilon(M-i)}(L(p))^2\mathbf{P}_{1/2}[\mathcal
{A}_4(1,L(p))]~~\mbox{ by (\ref{e9})}\nonumber\\
&\leq C_{12}(L(p))^2\mathbf{P}_{1/2}[\mathcal
{A}_4(1,L(p))].\label{e92}
\end{align}
Finally, by integrating (\ref{e73}) from $1/2$ to $p$ and using
(\ref{e92}) and (\ref{e8}), we have
\begin{equation*}
\mathbf{E}_{1/2}[T(p)]-\mathbf{E}_p[T(p)]\leq
C_{13}+C_{12}(p-1/2)(L(p))^2\mathbf{P}_{1/2}[\mathcal
{A}_4(1,L(p))]\leq C,
\end{equation*}
which completes the proof.
\end{proof}
\begin{remark}
Let us mention that without using Lemma 1 one can derive (\ref{e70})
by a weaker result.  It is noted just below (2.8) of \cite{BN11}, by
using a theorem from Reimer (Theorem 3 in \cite{BN11}), one easily
obtains $\mathbf{P}_{1/2}[\mathcal {A}_{(111111)}(2^{i+1},2^K)]\leq
\mathbf{P}_{1/2}[\mathcal {A}_{(011111)}(2^{i+1},2^K)]$.  This
inequality together with $\mathbf{P}_{1/2}[\mathcal
{A}_6(2^{i+1},2^K)]\leq \mathbf{P}_{1/2}[\mathcal
{A}_4(2^{i+1},2^K)]2^{-\varepsilon(K-i-1)}$ enables us to derive
(\ref{e70}) by a more complicated argument.  We will not give the
details here.
\end{remark}

Let us now bound the second term of (\ref{e74}).
\begin{lemma}\label{l8}
There is a constant $C>0$ such that for all $p>1/2$,
\begin{equation*}
\mathbf{E}_p[T(0,\mathcal {C}_{\infty})-T(p)]\leq C.
\end{equation*}
\end{lemma}
\begin{proof}
By (\ref{e11}), (\ref{e40}) and FKG, there exist $C_1,C_2>0$, such
that for all $p>1/2$ and $R\geq L(p)$,
\begin{equation}\label{e12}
\mathbf{P}_{p}[\mathcal {O}(R,2R)\cap \mathcal {A}_1(R,\infty)]\geq
1-C_1\exp\left(-C_2\frac{R}{L(p)}\right).
\end{equation}
For $j\geq 0$, define the event
\begin{equation*}
\mathcal {W}_j:=\{j=\min\{i\geq 0: \mathcal
{O}(2^{i}L(p),2^{i+1}L(p))\cap\mathcal {A}_1(2^{i}L(p),\infty)\mbox{
occurs}\}\}.
\end{equation*}
Then we have
\begin{align*}
&\mathbf{E}_p[T(0,\mathcal {C}_{\infty})-T(p)]\\
&\leq\sum_{j=0}^{\infty}\mathbf{E}_p[S(L(p),2^{j+1}L(p))\mathbf{1}_{\mathcal {W}_j}]\\
&\leq
\sum_{j=0}^{\infty}\left\{\mathbf{E}_p[S^2(L(p),2^{j+1}L(p))]\right\}^{1/2}\left\{\mathbf{P}_p[\mathcal
{W}_j]\right\}^{1/2}~~\mbox{ by Cauchy-Schwarz inequality}\\
&\leq \sum_{j=0}^{\infty}C_3(j+1)\exp(-C_22^j)~~\mbox{ by
(\ref{e12}) and Lemma \ref{l10}},
\end{align*}
which concludes the proof.
\end{proof}
Combining Lemmas \ref{l7} and \ref{l8}, the two terms on the
right-hand side of (\ref{e74}) are bounded, and we obtain
(\ref{t31}).
\subsection{Study of the variance}\label{var}
Suppose $p>1/2$.  Recall $T(p):=T(0,\partial B(L(p)))$.  Let
$J(p):=\lceil\log_2(L(p))\rceil+1$.  To prove (\ref{t35}), we write
\begin{align}
&|\Var_{1/2}[T(p)]-\Var_p[T(0,\mathcal {C}_{\infty})]|\nonumber\\
&\leq|\Var_{1/2}[T(p)]-\Var_p[T(p)]|+|\Var_p[T(0,\mathcal
{C}_{\infty})]-\Var_p[T(0,\mathcal
{C}_{J(p)})]|\nonumber\\
&\quad+|\Var_p[T(0,\mathcal {C}_{J(p)})]-\Var_p[T(p)]|,\label{e75}
\end{align}
where $\mathcal {C}_{J(p)}$ is an open circuit surrounding 0 and is
defined in Section \ref{term2}.  In Sections \ref{term1},
\ref{term2} and \ref{term3}, we will bound the three terms on the
right-hand side of (\ref{e75}), respectively.  Let us now focus on
the first term.
\subsubsection{Bound on
$|\Var_{1/2}[T(p)]-\Var_p[T(p)]|$}\label{term1}
\begin{lemma}\label{l11}
There is a constant $C>0$ such that for all $p>1/2$,
\begin{equation*}
|\Var_{1/2}[T(p)]-\Var_p[T(p)]|\leq C.
\end{equation*}
\end{lemma}
\begin{proof}
Similarly to the proof of Lemma \ref{l7}, we shall use Russo's
formula again, although the proof turns out to be more involved.
Recall the definition of the event $\mathcal {E}_v$ defined in the
proof of Lemma \ref{l7}.  For $1/2\leq h\leq p$, applying Lemma
\ref{l6}, one obtains
\begin{align}
&\frac{d}{dh}\{\mathbf{E}_h[T^2(p)]-\mathbf{E}_h^2[T(p)]\}\nonumber\\
&=\sum_{v\in B(L(p))}\{\mathbf{E}_h [\delta_v T^2(p)]-2\mathbf{E}_h[T(p)]\mathbf{E}_h [\delta_{v}T(p)]\}\nonumber\\
&=\sum_{v\in B(L(p))}\{\mathbf{E}_h [T^2(p)(\omega^v)-T^2(p)(\omega_v)]-2\mathbf{E}_h[T(p)]\mathbf{E}_h [T(p)(\omega^v)-T(p)(\omega_v)]\}\nonumber\\
&=\sum_{v\in
B(L(p))}\{\mathbf{E}_h[\mathbf{1}_{\mathcal {E}_v}]+2\mathbf{E}_h[\mathbf{1}_{\mathcal {E}_v}]\mathbf{E}_h[T(p)]-2\mathbf{E}_h[\mathbf{1}_{\mathcal {E}_v}T(p)(\omega_v)]\}\nonumber\\
&=\sum_{v\in B(L(p))}(1-2h)\mathbf{P}_h[\mathcal {E}_v]+\sum_{v\in
B(L(p))}2\{\mathbf{P}_h[\mathcal
{E}_v]\mathbf{E}_h[T(p)]-\mathbf{E}_h[\mathbf{1}_{\mathcal
{E}_v}T(p)]\}.\label{e78}
\end{align}
Unlike the expectation, it is not clear if $\Var_{h}[T(p)]$ is
monotonic in $h$.  So we need to bound the absolute value of the
above derivative.  It turns out that the key ingredient is to prove
the following claim: For $v\in B(4,L(p))$,
\begin{equation}\label{e25}
\mathbf{E}_h[\mathbf{1}_{\mathcal {E}_v}T(p)]=\mathbf{P}_h[\mathcal
{E}_v]\mathbf{E}_h[T(p)]+O(1)\mathbf{P}_{1/2}[\mathcal
{A}_4(1,|v|)].
\end{equation}

To show the claim (\ref{e25}), we will control the decorrelation of
$\mathbf{1}_{\mathcal {E}_v}T(p)$ and give the lower and upper
bounds of $\mathbf{E}_h[\mathbf{1}_{\mathcal {E}_v}T(p)]$
separately.  We start with some notation.  Assume $v\in
A(4,L(p)/2)$.  Define the events
\begin{align*}
&\mathcal {G}_i^-:=\left\{
\begin{aligned}
&\{i=\min\{k\geq 1:\mathcal {O}(2^{-k-1}|v|,2^{-k}|v|)\mbox{ occurs}\}\}& \mbox{ if }1\leq i\leq \lfloor\log_2|v|\rfloor-1,\\
&\{\forall~1\leq k\leq i-1,\mathcal {O}(2^{-k-1}|v|,2^{-k}|v|)\mbox{
does not occur}\} &\mbox{if }i=\lfloor\log_2|v|\rfloor.
\end{aligned}
\right.\\
&\mathcal {G}_j^+:=\left\{
\begin{aligned}
&\{j=\min\{k\geq 1:\mathcal {O}(2^k|v|,2^{k+1}|v|)
\mbox{ occurs}\}\}&\mbox{ if }1\leq j\leq \lfloor\log_2(L(p)/|v|)\rfloor-1,\\
&\{\forall~1\leq k\leq j-1, \mathcal {O}(2^k|v|,2^{k+1}|v|)\mbox{
does not occur}\} &\mbox{ if }j=\lfloor\log_2(L(p)/|v|)\rfloor.
\end{aligned}
\right.\\
&\mathcal {G}_{i,j}:=\mathcal {G}_i^-\cap\mathcal {G}_j^+.
\end{align*}
Write $\mathcal {I}=\mathcal
{I}(|v|):=\{1,2,\ldots,\lfloor\log_2|v|\rfloor\}$ and $\mathcal
{J}=\mathcal
{J}(p,|v|):=\{1,2,\ldots,\lfloor\log_2(L(p)/|v|)\rfloor\}$.  By RSW
and FKG, it is standard to show that there exist universal
$C_1,C_2>0$ such that
\begin{equation}\label{e20}
\exp(-C_1(i+j))\leq \mathbf{P}_h[\mathcal {G}_{i,j}]\leq
\exp(-C_2(i+j)).
\end{equation}

Let $\widehat{\mathcal {A}}_4(v;1,r)$ be the event defined in the
proof of Lemma \ref{l7}.  For $v\in B(4,L(p))$, define the event
\begin{align*}
\widetilde{\mathcal {E}}_v:=\left\{
\begin{aligned}
&\{\exists~1\leq r\leq |v|/3\mbox{ s.t. $\mathcal {A}_4(v;1,r)\cap
\mathcal {A}_{(1111)}(v;r,|v|/3)$
occurs}\} &\mbox{if }v\in B(4,L(p)/2),\\
&\{\exists~1\leq r\leq |v|/3\mbox{ s.t. $\widehat{\mathcal
{A}}_4(v;1,r)\cap \mathcal {A}_{(1111)}(v;r,|v|/3)$ occurs}\}
&\mbox{if }v\in B(L(p)/2,L(p)).
\end{aligned}
\right.
\end{align*}
Note that $\mathcal {E}_v\subset \widetilde{\mathcal {E}}_v$.
Similarly to the proof of (\ref{e70}) and (\ref{e72}),  one derives
that there is a universal $C_3>0$, such that
\begin{equation}\label{e21}
\mathbf{P}_h[\widetilde{\mathcal {E}}_v]\leq
C_3\mathbf{P}_{1/2}[\mathcal {A}_4(1,|v|)].
\end{equation}

For $v\in A(4,L(p)/2)$, write
\begin{equation*}
\mathbf{E}_h[\mathbf{1}_{\mathcal {E}_v}T(p)]=\sum_{i\in \mathcal
{I},j\in\mathcal {J}}\mathbf{P}_h[\mathcal
{G}_{i,j}]\mathbf{E}_h[\mathbf{1}_{\mathcal {E}_v}T(p)|\mathcal
{G}_{i,j}].
\end{equation*}
For convenience, let $T(0,\partial B(r))=0$ if $r\leq\sqrt{3}/2$ and
let $T(\partial B(r),\partial B(R))=0$ if $r\geq R$.  Observe that
conditioned on $\mathcal {G}_{i,j}$, the indicator function
$\mathbf{1}_{\mathcal {E}_v}$ and $T(0,\partial
B(2^{-i-1}|v|))+T(\partial B(2^{j+1}|v|),\partial B(L(p)))$ are
independent.  Indeed, conditionally on $\mathcal {G}_{i,j}$, we
distinguish the following four cases: (1) Suppose that $i\in
\mathcal {I}\backslash\{\lfloor\log_2|v|\rfloor\}$ and $j\in
\mathcal {J}\backslash\{\lfloor\log_2(L(p)/|v|)\rfloor\}$.  The
event $\mathcal {E}_v$ is equivalent to the event that for
$\omega_v$, there exists a geodesic $\gamma$ from an open circuit
surrounding 0 in $A(2^{-i-1}|v|,2^{-i}|v|)$ to an open circuit
surrounding 0 in $A(2^j|v|,2^{j+1}|v|)$, with $v\in\gamma$.  So
$\mathcal {E}_v$ is independent of the configuration outside
$A(2^{-i-1}|v|,2^{j+1}|v|)$.  In particular, $\mathcal {E}_v$ is
independent of $T(0,\partial B(2^{-i-1}|v|))+T(\partial
B(2^{j+1}|v|),\partial B(L(p)))$.  (2) Suppose that $i\in \mathcal
{I}\backslash\{\lfloor\log_2|v|\rfloor\}$ and
$j=\lfloor\log_2(L(p)/|v|)\rfloor$.  The event $\mathcal {E}_v$ is
equivalent to the event that for $\omega_v$, there exists a geodesic
$\gamma$ from an open circuit surrounding 0 in
$A(2^{-i-1}|v|,2^{-i}|v|)$ to $\partial B(L(p))$, with $v\in\gamma$.
So $\mathcal {E}_v$ is independent of $T(0,\partial
B(2^{-i-1}|v|))$.  (3) Suppose that $i=\lfloor\log_2|v|\rfloor$ and
$j\in \mathcal {J}\backslash\{\lfloor\log_2(L(p)/|v|)\rfloor\}$. The
argument for this case is similar to that for the case (2).  (4)
Suppose that $i=\lfloor\log_2|v|\rfloor$ and
$j=\lfloor\log_2(L(p)/|v|)\rfloor$.  Then the observation is trivial
since $T(0,\partial B(2^{-i-1}|v|))=T(\partial
B(2^{j+1}|v|),\partial B(L(p)))=0$ in this case.

Then we have
\begin{align*}
&\mathbf{E}_h[\mathbf{1}_{\mathcal {E}_v}T(p)]\\
&\geq\sum_{i\in \mathcal {I},j\in\mathcal {J}}\mathbf{P}_h[\mathcal {G}_{i,j}]\mathbf{E}_h[\mathbf{1}_{\mathcal {E}_v}|\mathcal {G}_{i,j}]\mathbf{E}_h[T(0,\partial B(2^{-i-1}|v|))+T(\partial B(2^{j+1}|v|),\partial B(L(p)))]\\
&\geq\sum_{i\in \mathcal {I},j\in\mathcal {J}}\mathbf{P}_h[\mathcal
{E}_v\mathcal {G}_{i,j}]\{\mathbf{E}_h[T(p)]-\mathbf{E}_h[S(2^{-i-1}|v|,2^{j+1}|v|)]\}\\
&\geq\mathbf{P}_h[\mathcal {E}_v]\mathbf{E}_h[T(p)]-\sum_{i\in
\mathcal {I},j\in\mathcal {J}}C_4(i+j)\mathbf{P}_h[\mathcal
{E}_v\mathcal {G}_{i,j}]~~\mbox{ by Lemma \ref{l10}}.
\end{align*}
Since $\mathcal {E}_v\subset \widetilde{\mathcal {E}}_v$ and
$\widetilde{\mathcal {E}}_v$ is independent of the event $\mathcal
{G}_{i,j}$, we have
\begin{align*}
\mathbf{E}_h[\mathbf{1}_{\mathcal {E}_v}T(p)]&\geq
\mathbf{P}_h[\mathcal {E}_v]\mathbf{E}_h[T(p)]-\sum_{i\in \mathcal
{I},j\in\mathcal
{J}}C_4(i+j)\exp(-C_2(i+j))\mathbf{P}_h[\widetilde{\mathcal
{E}}_v]~~\mbox{
by (\ref{e20})}\\
&\geq \mathbf{P}_h[\mathcal
{E}_v]\mathbf{E}_h[T(p)]-C_5\mathbf{P}_{1/2}[\mathcal
{A}_4(1,|v|)]~~\mbox{ by (\ref{e21})},
\end{align*}
which gives the desired lower bound of
$\mathbf{E}_h[\mathbf{1}_{\mathcal {E}_v}T(p)]$.  To get the upper
bound, we need more notation. For $i\in \mathcal {I},j\in\mathcal
{J}$, define
\begin{align*}
&X_v:=\mbox{the maximal number of disjoint closed paths
connecting the two boundary pieces}\\
&\qquad\quad\mbox{of
$A(v;|v|/3,|v|/2)$},\\
&Y_v(i,j):=\mbox{the maximal number of disjoint closed circuits
surrounding 0 and intersecting}\\
&\qquad\qquad\quad\mbox{ $A(2^{-i-1}|v|,2^{j+1}|v|)\backslash
B(v;|v|/3)$}.
\end{align*}
Observe that
\begin{equation}\label{e77}
X_v+Y_v(i,j)\geq S(2^{-i-1}|v|,2^{j+1}|v|).
\end{equation}
It is clear that $X_v$ is independent of $\mathcal {G}_{i,j}$.
Then, using RSW and BK inequality, it is standard that
\begin{equation}\label{e76}
\mathbf{E}_h[X_v|\mathcal {G}_{i,j}]=\mathbf{E}_h[X_v]\leq
\mathbf{E}_{1/2}[X_v]\leq C_6.
\end{equation}
Since $Y_v(i,j)\leq S(2^{-i-1}|v|,2^{j+1}|v|)$, we deduce that there
are universal constants $C_7,C_8>0$, such that for all $x\geq
K(i+j+2)$, where $K$ is from Lemma \ref{l10},
\begin{align*}
\mathbf{P}_h[Y_v(i,j)\geq x|\mathcal
{G}_{i,j}]&\leq\frac{\mathbf{P}_h[S(2^{-i-1}|v|,2^{j+1}|v|)\geq
x]}{\mathbf{P}_h[\mathcal {G}_{i,j}]}\\
&\leq C_7\exp(-C_8x)\exp(C_1(i+j))\mbox{ by Lemma \ref{l10} and
(\ref{e20}),}
\end{align*}
which implies that there is a universal $C_9>0$, such that
\begin{equation}\label{e17}
\mathbf{E}_h[Y_v(i,j)|\mathcal {G}_{i,j}]\leq C_9(i+j).
\end{equation}
Then for $v\in A(4,L(p)/2)$, we get the desired upper bound as
follows.
\begin{align*}
&\mathbf{E}_h[\mathbf{1}_{\mathcal {E}_v}T(p)]\\
&\leq\sum_{i\in \mathcal {I},j\in\mathcal {J}}\mathbf{P}_h[\mathcal {E}_v,\mathcal {G}_{i,j}]\mathbf{E}_h[T(0,\partial B(2^{-i-1}|v|))+T(\partial B(2^{j+1}|v|),\partial B(L(p)))]\\
&\qquad+\sum_{i\in \mathcal {I},j\in\mathcal
{J}}\mathbf{P}_h[\widetilde{\mathcal {E}}_v]\mathbf{P}_h[\mathcal
{G}_{i,j}]\mathbf{E}_h[X_v+Y_v(i,j)|\mathcal {G}_{i,j}]~~\mbox{
by (\ref{e77})}\\
&\leq \mathbf{P}_h[\mathcal {E}_v]\mathbf{E}_h[T(p)]+\sum_{i\in
\mathcal {I},j\in\mathcal
{J}}C_{10}(i+j)\exp(-C_2(i+j))\mathbf{P}_h[\widetilde{\mathcal
{E}}_v]~~\mbox{
by (\ref{e20}), (\ref{e76}) and (\ref{e17})}\\
&\leq \mathbf{P}_h[\mathcal
{E}_v]\mathbf{E}_h[T(p)]+C_{11}\mathbf{P}_{1/2}[\mathcal
{A}_4(1,|v|)]~~\mbox{ by (\ref{e21})}.
\end{align*}
The above lower and upper bounds yield (\ref{e25}) for $v\in
A(4,L(p)/2)$.  Using (\ref{e21}), the proof for the case of $v\in
B(L(p)/2,L(p))$ is very similar to the case of $v\in B(4,L(p)/2)$;
one needs to obtain lower and upper bounds of
$\mathbf{E}_h[\mathbf{1}_{\mathcal {E}_v}T(p)]$ as above.  So the
proof is omitted here.  Thus our claim (\ref{e25}) is established.
For $v\in A(1,4)$, the proof of the following equation (\ref{e22})
is much simpler than that of (\ref{e25}).  The proof is also
omitted.
\begin{equation}\label{e22}
\mathbf{E}_h[\mathbf{1}_{\mathcal
{E}_v}T(p)]=\mathbf{E}_h[\mathbf{1}_{\mathcal
{E}_v}]\mathbf{E}_h[T(p)]+O(1).
\end{equation}
Combining (\ref{e78}), (\ref{e25}), (\ref{e22}) and the proof of
(\ref{e92}), we have
\begin{align*}
\left|\frac{d}{dh}\{\mathbf{E}_h[T^2(p)]-\mathbf{E}_h^2
[T(p)]\}\right|&\leq C_{12}+C_{13}\sum_{v\in
A(2,L(p))}\mathbf{P}_{1/2}[\mathcal {A}_4(1,|v|)]\\
&\leq C_{12}+C_{14}(L(p))^2\mathbf{P}_{1/2}[\mathcal {A}_4(1,L(p))].
\end{align*}
Finally, by integrating over the interval $[1/2,p]$ and using
(\ref{e8}) we obtain the desired result.
\end{proof}
\subsubsection{Bound on $|\Var_p[T(0,\mathcal
{C}_{\infty})]-\Var_p[T(0,\mathcal {C}_{J(p)})]|$}\label{term2} We
now wish to bound the second term on the right-hand side of
(\ref{e75}). We will use the martingale method introduced in
\cite{KZ97}.  This approach has been used in \cite{DLW17,Yao18}
also.  We start with some notation.

For $j\in \mathbb{N}\cup\{0\}$, we write $A(j):=A(2^j,2^{j+1})$.
Define
\begin{align*}
m(j)&:=\inf\{k\geq j: A(k)\mbox{ contains an open circuit
surrounding 0}\},\\
\mathcal {C}_j&:=\mbox{the innermost open circuit surrounding 0
in }A(m(j)),\\
\mathscr{F}_j&:=\sigma\mbox{-field generated by }\{t(v):v\in
\overline{\mathcal {C}}_j\}.
\end{align*}
Denote by $\mathcal {C}_{-1}$ the origin and by $\mathscr{F}_{-1}$
the trivial $\sigma$-field.  For $p\geq 1/2$ and $q\in \mathbb{N}$,
write
\begin{align*}
T(0,\mathcal {C}_q)-\mathbf{E}_p[T(0,\mathcal
{C}_q)]=\sum_{j=0}^q\left(\mathbf{E}_p[T(0,\mathcal
{C}_q)|\mathscr{F}_j]-\mathbf{E}_p[T(0,\mathcal
{C}_q)|\mathscr{F}_{j-1}]\right):=\sum_{j=0}^q\Delta_j.
\end{align*}
Then $\{\Delta_j\}_{0\leq j\leq q}$ is an $\mathscr{F}_j$-martingale
increment sequence.  Hence,
\begin{equation}\label{e15}
\Var_p[T(0,\mathcal {C}_q)]=\sum_{j=0}^q\mathbf{E}_p[\Delta_j^2].
\end{equation}

Let $(\Omega',\mathscr{F}',\mathbf{P}_p')$ be a copy of
$(\Omega,\mathscr{F},\mathbf{P}_p)$.  Denote by $\mathbf{E}_p'$ the
expectation with respect to $\mathbf{P}_p'$, and by $\omega'$ a
sample point in $\Omega'$. Let $T(\cdot,\cdot)(\omega),m(j,\omega)$
and $\mathcal {C}_j(\omega)$ denote the quantities defined before,
but with explicit dependence on $\omega$. Define
$l(j,\omega,\omega'):=m(m(j,\omega)+1,\omega')$.  We need the
following lemma, which was essentially proved in \cite{KZ97}.  Note
that (\ref{e28}) is standard and follows from RSW and FKG;
(\ref{e35}) is the same as Lemma 2 of \cite{KZ97}.
\begin{lemma}[\cite{KZ97}]\label{l14}
(i) There exists $C>0$, such that for all $j,k\in\mathbb{N}$ and
$p\geq 1/2$, we have
\begin{equation}\label{e28}
\mathbf{P}_p[m(j)\geq j+k]\leq \exp(-Ck).
\end{equation}
(ii) For $j\geq 0$, $\Delta_j$ does not depend on $q$. Furthermore,
\begin{align}
\Delta_j(\omega)=&T(\mathcal {C}_{j-1}(\omega),\mathcal
{C}_j(\omega))(\omega)+\mathbf{E}_p'[T(\mathcal
{C}_j(\omega),\mathcal
{C}_{l(j,\omega,\omega')}(\omega'))(\omega')]\nonumber\\
&-\mathbf{E}_p'[T(\mathcal {C}_{j-1}(\omega),\mathcal
{C}_{l(j,\omega,\omega')}(\omega'))(\omega')].\label{e35}
\end{align}
\end{lemma}
Similarly to (\ref{e15}), the next lemma allows us to express the
variance of $T(0,\mathcal {C}_{\infty})$ in terms of sums of
$\mathbf{E}_p[\Delta_j^2]$.
\begin{lemma}\label{l13}
Suppose $p>1/2$.  We have
\begin{align}
&\mathbf{E}_p[T^2(0,\mathcal {C}_{\infty})]<\infty,\label{e18}\\
&\lim_{q\rightarrow\infty}\Var_p[T(0,\mathcal
{C}_q)]\rightarrow\Var_p[T(0,\mathcal {C}_{\infty})],\label{e19}\\
&\Var_p[T(0,\mathcal
{C}_{\infty})]=\sum_{j=0}^{\infty}\mathbf{E}_p[\Delta_j^2].\label{e27}
\end{align}
\end{lemma}
\begin{proof}
It is clear that (\ref{e18}) follows from (\ref{e11}).  Observe that
almost surely for all $q\geq 1$, $T(0,\mathcal {C}_q)\leq
T(0,\mathcal {C}_{\infty})$.  So for all $q\geq 1$,
\begin{equation*}
\mathbf{E}_p[T^2(0,\mathcal {C}_q)]\leq \mathbf{E}_p[T^2(0,\mathcal
{C}_{\infty})]<\infty.
\end{equation*}

For $j\geq \lceil \log_2(L(p))\rceil$, define the event
\begin{equation*}
\mathcal {W}_j:=\{j=\min\{i\geq \lceil \log_2(L(p))\rceil: \mathcal
{O}(2^{i},2^{i+1})\cap\mathcal {A}_1(2^{i},\infty)\mbox{
occurs}\}\}.
\end{equation*}
It is standard that $\bigcup_{j=\lceil
\log_2(L(p))\rceil}^{\infty}\mathcal {W}_j$ occurs with probability
one.  For $q\geq \lceil \log_2(L(p))\rceil$, there exist absolute
constants $C_1,C_2>0$, such that
\begin{align*}
&\mathbf{E}_p[(T(0,\mathcal {C}_q)-T(0,\mathcal {C}_{\infty}))^2]\\
&\leq\sum_{j=q}^{\infty}\mathbf{E}_p[\rho^2(2^q,2^{j+1})\mathbf{1}_{\mathcal {W}_j}]\\
&\leq
\sum_{j=q}^{\infty}\left\{\mathbf{E}_p[\rho^4(2^q,2^{j+1})]\right\}^{1/2}\left\{\mathbf{P}_p[\mathcal
{W}_j]\right\}^{1/2}~~\mbox{ by Cauchy-Schwarz inequality}\\
&\leq
\sum_{j=q}^{\infty}C_1(j+1-q)^2\exp\left(-C_2\frac{2^j}{L(p)}\right)~~\mbox{
by (\ref{e12}) and Lemma \ref{l10}}.
\end{align*}
This implies
\begin{equation}\label{e29}
\mathbf{E}_p[(T(0,\mathcal {C}_q)-T(0,\mathcal
{C}_{\infty}))^2]\rightarrow 0~~ \mbox{ as }q\rightarrow\infty.
\end{equation}
The triangle inequality for the norm
$\|\cdot\|_2=\sqrt{\mathbf{E}[|\cdot|^2]}$ and (\ref{e29}) give
\begin{align*}
\left|\sqrt{\Var_p[T(0,\mathcal {C}_q)]}-\sqrt{\Var_p[T(0,\mathcal
{C}_{\infty})]}\right|&\leq \sqrt{\Var_p[T(0,\mathcal
{C}_q)-T(0,\mathcal {C}_{\infty})]}\\
&\leq \sqrt{\mathbf{E}_p[(T(0,\mathcal {C}_q)-T(0,\mathcal
{C}_{\infty}))^2]}\rightarrow 0~~\mbox{ as }q\rightarrow\infty.
\end{align*}
This yields (\ref{e19}) immediately.

Item (ii) in Lemma \ref{l14} tells us that for $j\geq 0$, $\Delta_j$
does not depend on $q$.  This fact together with (\ref{e15}) and
(\ref{e19}) gives (\ref{e27}).
\end{proof}

We bound the second term on the right-hand side of (\ref{e75}) in
the following lemma.
\begin{lemma}\label{l15}
There exist universal constants $C_1,\ldots,C_4>0$, such that for
all $p>1/2$,
\begin{align}
&\mathbf{E}_p[\Delta_j^2]\leq C_1\quad\mbox{for all }j\geq 0,\label{e36}\\
&\mathbf{E}_p[\Delta_{J(p)+j}^2]\leq C_2\exp(-C_32^j)\quad\mbox{for all }j\geq 0,\label{e31}\\
&|\Var_p[T(0,\mathcal {C}_{\infty})]-\Var_p[T(0,\mathcal
{C}_{J(p)})]|<C_4.\label{e32}
\end{align}
\end{lemma}
\begin{proof}
The proof of (\ref{e36}) is essentially the same as that of (29) in
\cite{Yao18}.  For completeness we give it here.  Assume that $j\geq
1$. Applying Lemma \ref{l10} and (\ref{e28}), there exist
$C_5,C_6>0$, such that for all $x\geq K+3$, where $K$ is from Lemma
\ref{l10},
\begin{align*}
&\mathbf{P}_p[|T(\mathcal {C}_{j-1},\mathcal
{C}_j)|\geq x]\\
&\leq \mathbf{P}_p[m(j)\geq j+\lfloor
x/K\rfloor-2]+\mathbf{P}_p[\rho(2^{j-1},2^{j+\lfloor
x/K\rfloor-1})\geq x]\leq C_5\exp(-C_6x).
\end{align*}
Similarly we have
\begin{equation*}
\mathbf{P}_p'[T(\mathcal {C}_j(\omega),\mathcal
{C}_{l(j,\omega,\omega')}(\omega'))(\omega')\geq x]\leq
C_7\exp(-C_8x),
\end{equation*}
which implies
\begin{equation*}
\mathbf{E}_p'[T^2(\mathcal {C}_j(\omega),\mathcal
{C}_{l(j,\omega,\omega')}(\omega'))(\omega')]\leq C_9.
\end{equation*}
Using the same method we obtain
\begin{equation*}
\mathbf{E}_p'[T^2(\mathcal {C}_{j-1}(\omega),\mathcal
{C}_{l(j,\omega,\omega')}(\omega'))(\omega')]\leq C_{10}.
\end{equation*}
Equation (\ref{e35}) together with above inequalities implies
(\ref{e36}) for $j\geq 1$ immediately.  The proof for the case that
$j=0$ is essentially the same.

To prove (\ref{e31}), we will use (\ref{e35}) to get large deviation
estimates of $|\Delta_{J(p)+j}|$.  For all $j\geq 0,k\geq 1$ and
$n>2^{J(p)+j-1}$, define the events
\begin{align*}
&\mathcal {Q}_n(j,k):=\{\mbox{$\exists$ $k$ disjoint closed
circuits surrounding 0 in $A(2^{J(p)+j-1},n)$}\},\\
&\mathcal {Q}(j,k):=\{\mbox{$\exists$ $k$ disjoint closed circuits
surrounding 0 and lying outside $B(2^{J(p)+j-1})$}\}.
\end{align*}
By BK inequality (with the condition that the events depend on
finitely many sites) and (\ref{e11}),
\begin{equation*}
\mathbf{P}_p[\mathcal {Q}_n(j,k)]\leq \{\mathbf{P}_p[\mathcal
{Q}_n(j,1)]\}^k\leq\{\mathbf{P}_p[\mathcal
{A}_1^c(2^{J(p)+j-1},\infty)]\}^k\leq C_{11}^k\exp(-C_{12}2^jk).
\end{equation*}
Therefore, there is an absolute constant $j_0>1$, such that for all
$j\geq j_0$,
\begin{equation*}
\mathbf{P}_p[\mathcal
{Q}(j,k)]=\lim_{n\rightarrow\infty}\mathbf{P}_p[\mathcal
{Q}_n(j,k)]\leq \exp(-C_{13}2^jk).
\end{equation*}
This implies that for all $j\geq j_0$,
\begin{align}
&\mathbf{P}_p[T(\mathcal {C}_{J(p)+j-1},\mathcal {C}_{J(p)+j})\geq
k]\leq \mathbf{P}_p[\mathcal {Q}(j,k)]\leq
\exp(-C_{13}2^jk),\label{e37}\\
&\mathbf{P}_p'[T(\mathcal {C}_{J(p)+j-1}(\omega),\mathcal
{C}_{l(J(p)+j,\omega,\omega')}(\omega'))(\omega')\geq
k]\leq\mathbf{P}_p[\mathcal {Q}(j,k)]\leq
\exp(-C_{13}2^jk).\nonumber
\end{align}
From the above inequality we know that there exists a constant
$j_1>j_0$, such that for all $j\geq j_1$,
\begin{equation}\label{e38}
\mathbf{E}_p'[T(\mathcal {C}_{J(p)+j-1}(\omega),\mathcal
{C}_{l(J(p)+j,\omega,\omega')}(\omega'))(\omega')]\leq
\exp(-C_{14}2^j),
\end{equation}
which obviously implies
\begin{equation}\label{e39}
\mathbf{E}_p'[T(\mathcal {C}_{J(p)+j}(\omega),\mathcal
{C}_{l(J(p)+j,\omega,\omega')}(\omega'))(\omega')]\leq
\exp(-C_{14}2^j).
\end{equation}
Equation (\ref{e35}) together with (\ref{e37}), (\ref{e38}) and
(\ref{e39}) implies that there is an absolute constant $j_2>j_1$,
such that for all $p>1/2$, $j\geq j_2$ and $x\geq 0$,
\begin{equation*}
\mathbf{P}_p[|\Delta_{J(p)+j}|\geq \exp(-C_{14}2^j)+x]\leq
\exp(-C_{13}2^jx),
\end{equation*}
which yields $\mathbf{E}_p[\Delta_{J(p)+j}^2]\leq
C_{15}\exp(-C_{16}2^j)$ for $j\geq j_2$.  This and (\ref{e36}) give
(\ref{e31}) immediately.

Combining (\ref{e15}), (\ref{e27}) and (\ref{e31}), we obtain
(\ref{e32}) as follows.
\begin{equation*}
|\Var_p[T(0,\mathcal {C}_{\infty})]-\Var_p[T(0,\mathcal
{C}_{J(p)})]|=\sum_{j=1}^{\infty}\mathbf{E}_p[\Delta_{J(p)+j}^2]\leq
\sum_{j=1}^{\infty}C_2\exp(-C_32^j)<C_4.
\end{equation*}
\end{proof}
\subsubsection{Bound on $|\Var_p[T(0,\mathcal
{C}_{J(p)})]-\Var_p[T(p)]|$}\label{term3} The following lemma gives
the upper bound of the third term on the right-hand side of
(\ref{e75}).
\begin{lemma}\label{l16}
There exists a universal constant $C>0$, such that for all $p>1/2$,
\begin{equation*}
|\Var_p[T(0,\mathcal {C}_{J(p)})]-\Var_p[T(p)]|<C.
\end{equation*}
\end{lemma}
We have already known that $\Var_p[T(0,\mathcal
{C}_{J(p)})]=\sum_{j=0}^{J(p)}\mathbf{E}_p[\Delta_j^2]$ from
(\ref{e15}).  To prove the lemma, we will write
$T(p)-\mathbf{E}_p[T(p)]$ as a sum of martingale differences
$\widetilde{\Delta}_j$'s, and then bound
$|\mathbf{E}_p[\widetilde{\Delta}_j^2]-\mathbf{E}_p[\Delta_j^2]|$
appropriately.  Let us start with some notation.  For $j\in
\mathbb{N}\cup\{0\}$, let
\begin{align*}
&\widetilde{\mathcal {C}}_j:=\left\{\begin{array}{ll}
\mathcal {C}_j &\mbox{if $m(j)\leq J(p)-3$},\\
\partial B(L(p)) &\mbox{otherwise}.\end{array}\right.\\
&\widetilde{\mathscr{F}}_j:=\left\{\begin{array}{ll}
\sigma\mbox{-field generated by }\{t(v):v\in
\overline{\mathcal {C}}_j\} &\mbox{if $m(j)\leq J(p)-3$},\\
\sigma\mbox{-field generated by }\{t(v):v\in B(L(p))\}
&\mbox{otherwise}.\end{array}\right.
\end{align*}
Denote by $\widetilde{\mathscr{F}}_{-1}$ the trivial $\sigma$-field
and by $\widetilde{\mathcal {C}}_{-1}$ the origin.  For $p>1/2$,
write
\begin{align*}
T(p)-\mathbf{E}_p[T(p)]=\sum_{j=0}^{J(p)-2}\left(\mathbf{E}_p[T(p)|\widetilde{\mathscr{F}}_j]-\mathbf{E}_p[T(p)|\widetilde{\mathscr{F}}_{j-1}]\right):=\sum_{j=0}^{J(p)-2}\widetilde{\Delta}_j.
\end{align*}
Then $\{\widetilde{\Delta}_j\}_{0\leq j\leq J(p)-2}$ is an
$\widetilde{\mathscr{F}}_j$-martingale increment sequence. So,
\begin{equation}\label{e41}
\Var_p[T(p)]=\sum_{j=0}^{J(p)-2}\mathbf{E}_p[\widetilde{\Delta}_j^2].
\end{equation}
We claim that for all $j\geq 0$,
\begin{align}
\widetilde{\Delta}_j(\omega)=&T(\widetilde{\mathcal
{C}}_{j-1}(\omega),\widetilde{\mathcal
{C}}_j(\omega))(\omega)+\mathbf{E}_p'[T(\widetilde{\mathcal
{C}}_j(\omega),\widetilde{\mathcal
{C}}_{l(j,\omega,\omega')}(\omega'))(\omega')]\nonumber\\
&-\mathbf{E}_p'[T(\widetilde{\mathcal
{C}}_{j-1}(\omega),\widetilde{\mathcal
{C}}_{l(j,\omega,\omega')}(\omega'))(\omega')].\label{e42}
\end{align}
The proof is essentially the same as that of (\ref{e35}), and is
included in the Appendix.  Similarly to the proof of (\ref{e36}),
one can show that there is an absolute constant $C>0$ such that
\begin{equation}\label{e48}
\mathbf{E}_p[\widetilde{\Delta}_j^2]\leq C\quad\mbox{for all $p>1/2$
and all }j\geq 0.
\end{equation}

\begin{proof}[Proof of Lemma \ref{l16}]
Let $K$ be the constant from Lemma \ref{l10}.  By Lemma \ref{l10}
and (\ref{e28}), there exist $C_1,C_2>0$, such that for all $j,k\geq
1$,
\begin{align}
&\mathbf{P}_p[\rho(2^{j-1},2^{m(j)+1})\geq k]\nonumber\\
&\leq \mathbf{P}_p[m(j)\geq j-1+\lfloor
k/K\rfloor]+\mathbf{P}_p[\rho(2^{j-1},2^{j-1+\lfloor
k/K\rfloor})\geq k]\leq C_1\exp(-C_2k).\label{e81}
\end{align}
Suppose $0\leq j\leq J(p)-2$.  It is easy to see that if $m(j)\leq
J(p)-3$, then $T(\widetilde{\mathcal {C}}_{j-1},\widetilde{\mathcal
{C}}_j)=T(\mathcal {C}_{j-1},\mathcal {C}_j)$; if $m(j)\geq J(p)-2$,
then $|T(\widetilde{\mathcal {C}}_{j-1},\widetilde{\mathcal
{C}}_j)-T(\mathcal {C}_{j-1},\mathcal {C}_j)|\leq
\rho(2^{J(p)-2},2^{m(J(p)-2)+1})$.  Therefore, using (\ref{e28}),
(\ref{e81}) and independence, we obtain that there exist
$C_3,C_4>0$, such that for all $k\geq 1$,
\begin{align}
\mathbf{P}_p[|T(\widetilde{\mathcal {C}}_{j-1},\widetilde{\mathcal
{C}}_j)-T(\mathcal {C}_{j-1},\mathcal {C}_j)|\geq
k]&\leq\mathbf{P}_p[m(j)\geq
J(p)-2]\mathbf{P}_p[\rho(2^{J(p)-2},2^{m(J(p)-2)+1})\geq
k]\nonumber\\
&\leq C_3\exp(-C_4(J(p)-j+k)).\label{e47}
\end{align}
So for $j\leq J(p)-2$,
\begin{equation}\label{e82}
\mathbf{E}_p\left|T(\widetilde{\mathcal
{C}}_{j-1},\widetilde{\mathcal {C}}_j)-T(\mathcal {C}_{j-1},\mathcal
{C}_j)\right|^2\leq C_5\exp(-C_6(J(p)-j)).
\end{equation}
Similarly to (\ref{e47}), there exist $C_7,C_8>0$, such that for any
fixed $\mathcal {C}_j(\omega)$ with $m(j,\omega)\leq J(p)-3$ and all
$k\geq 1$,
\begin{align*}
&\mathbf{P}_p'[|T(\widetilde{\mathcal
{C}}_j(\omega),\widetilde{\mathcal
{C}}_{l(j,\omega,\omega')}(\omega'))(\omega')-T(\mathcal
{C}_j(\omega),\mathcal
{C}_{l(j,\omega,\omega')}(\omega'))(\omega')|\geq k]\\
&\leq\mathbf{P}_p'[m(m(j,\omega)+1,\omega')\geq
J(p)-2]\mathbf{P}_p'[\rho(2^{J(p)-2},2^{m(J(p)-2)+1})\geq
k]\\
&\leq C_7\exp(-C_8(J(p)-m(j,\omega)+k)),
\end{align*}
which gives
\begin{equation}\label{e43}
\left|\mathbf{E}_p'[T(\widetilde{\mathcal
{C}}_j(\omega),\widetilde{\mathcal
{C}}_{l(j,\omega,\omega')}(\omega'))(\omega')]-\mathbf{E}_p'[T(\mathcal
{C}_j(\omega),\mathcal
{C}_{l(j,\omega,\omega')}(\omega'))(\omega')]\right|\leq
C_9\exp(-C_8(J(p)-m(j,\omega))).
\end{equation}
Note that if $m(j,\omega)\geq J(p)-2$, then $\widetilde{\mathcal
{C}}_j(\omega)=\partial B(L(p))$ and $T(\widetilde{\mathcal
{C}}_j(\omega),\widetilde{\mathcal
{C}}_{l(j,\omega,\omega')}(\omega'))(\omega')=0$.  Therefore, for
any fixed $\mathcal {C}_j(\omega)$ with $m(j,\omega)\geq J(p)-2$ and
all $k\geq 1$,
\begin{align*}
&\mathbf{P}_p'[|T(\widetilde{\mathcal
{C}}_j(\omega),\widetilde{\mathcal
{C}}_{l(j,\omega,\omega')}(\omega'))(\omega')-T(\mathcal
{C}_j(\omega),\mathcal
{C}_{l(j,\omega,\omega')}(\omega'))(\omega')|\geq k]\\
&=\mathbf{P}_p'[T(\mathcal {C}_j(\omega),\mathcal
{C}_{l(j,\omega,\omega')}(\omega'))(\omega')\geq k]\\
&\leq \mathbf{P}_p'[\rho(2^{m(j,\omega)},2^{m(m(j,\omega)+1)+1})(\omega')\geq k]\\
&\leq C_1\exp(-C_2k)\quad\mbox{by (\ref{e81})},
\end{align*}
which gives
\begin{equation}\label{e44}
\left|\mathbf{E}_p'[T(\widetilde{\mathcal
{C}}_j(\omega),\widetilde{\mathcal
{C}}_{l(j,\omega,\omega')}(\omega'))(\omega')]-\mathbf{E}_p'[T(\mathcal
{C}_j(\omega),\mathcal
{C}_{l(j,\omega,\omega')}(\omega'))(\omega')]\right|\leq C_{10}.
\end{equation}
Combining (\ref{e43}) and (\ref{e44}) with (\ref{e28}), we obtain
that for $j\leq J(p)-3$,
\begin{align}
&\mathbf{E}_p\left|\mathbf{E}_p'[T(\widetilde{\mathcal
{C}}_j(\omega),\widetilde{\mathcal
{C}}_{l(j,\omega,\omega')}(\omega'))(\omega')]-\mathbf{E}_p'[T(\mathcal
{C}_j(\omega),\mathcal
{C}_{l(j,\omega,\omega')}(\omega'))(\omega')]\right|^2\nonumber\\
&\leq
\sum_{k=j}^{J(p)-3}\mathbf{P}_p[m(j)=k]C_9^2\exp(-2C_8(J(p)-k))+C_{10}^2\mathbf{P}_p[m(j)\geq J(p)-2]\nonumber\\
&\leq
\sum_{k=j}^{J(p)-3}\exp(-C_{11}(k-j))C_9^2\exp(-2C_8(J(p)-k))+C_{10}^2\exp(-C_{11}(J(p)-2-j))\nonumber\\
&\leq C_{12}\exp(-C_{13}(J(p)-j))\label{e45}.
\end{align}
Similarly, for $j\leq J(p)-3$,
\begin{equation}\label{e46}
\mathbf{E}_p\left|\mathbf{E}_p'[T(\widetilde{\mathcal
{C}}_{j-1}(\omega),\widetilde{\mathcal
{C}}_{l(j,\omega,\omega')}(\omega'))(\omega')]-\mathbf{E}_p'[T(\mathcal
{C}_{j-1}(\omega),\mathcal
{C}_{l(j,\omega,\omega')}(\omega'))(\omega')]\right|^2\leq
C_{14}\exp(-C_{15}(J(p)-j)).
\end{equation}
Combining the equations (\ref{e35}), (\ref{e42}) together with
inequalities (\ref{e82}), (\ref{e45}) and (\ref{e46}), we obtain
that, there exist absolute constants $C_{16},C_{17}>0$, such that
for all $0\leq j\leq J(p)-3$,
\begin{equation}\label{e49}
\mathbf{E}_p|\widetilde{\Delta}_j-\Delta_j|^2\leq
C_{16}\exp(-C_{17}(J(p)-j)).
\end{equation}
Therefore, for all $0\leq j\leq J(p)-3$,
\begin{align}
&|\mathbf{E}_p[\widetilde{\Delta}_j^2]-\mathbf{E}_p[\Delta_j^2]|\nonumber\\
&\leq\sqrt{\mathbf{E}_p[(\widetilde{\Delta}_j+\Delta_j)^2]}\sqrt{\mathbf{E}_p[(\widetilde{\Delta}_j-\Delta_j)^2]}\quad\mbox{by Cauchy-Schwarz inequality}\nonumber\\
&\leq C_{18}\exp(-C_{19}(J(p)-j))\quad\mbox{by (\ref{e36}),
(\ref{e48}) and (\ref{e49})}\label{e50}.
\end{align}
Finally, we have
\begin{align*}
|\Var_p[T(0,\mathcal
{C}_{J(p)})]-\Var_p[T(p)]|&=\left|\sum_{j=0}^{J(p)}\mathbf{E}_p[\Delta_j^2]-\sum_{j=0}^{J(p)-2}\mathbf{E}_p[\widetilde{\Delta}_j^2]\right|\quad\mbox{by
(\ref{e15}) and (\ref{e41})}\nonumber\\
&\leq
C_{20}+\sum_{j=0}^{J(p)-3}\left|\mathbf{E}_p[\widetilde{\Delta}_j^2]-\mathbf{E}_p[\Delta_j^2]\right|\quad\mbox{by
(\ref{e36}) and (\ref{e48})}\nonumber\\
&\leq C\quad\mbox{by (\ref{e50})},
\end{align*}
which concludes the proof.
\end{proof}
By Lemma \ref{l11}, (\ref{e32}) and Lemma \ref{l16}, the three terms
on the right-hand side of (\ref{e75}) are bounded, and we obtain
(\ref{t35}).
\subsection{Proof of Corollary \ref{c1}}\label{cor}
\begin{proof}[Proof of Corollary \ref{c1}]
First we prove (\ref{t32}).  Let $C$ be the constant in (\ref{t31}).
For $k\in\mathbb{N}$, let $p_k$ be the solution of the equation
\begin{equation*}
-\log (p_k-1/2)=k^4C^2.
\end{equation*}
By (\ref{t41}) and (\ref{e14}), we have
\begin{equation}\label{e84}
\frac{T_{1/2}(0,\partial
B(L(p_k)))}{-\frac{4}{3}\log(p_k-1/2)}\rightarrow
\frac{1}{2\sqrt{3}\pi}~~\mbox{ a.s. as }k\rightarrow\infty.
\end{equation}
Define the event
\begin{equation*}
\mathcal {F}_k:=\left\{|T_{p_k}(0,\mathcal
{C}_{\infty}(p_k))-T_{1/2}(0,\partial B(L(p_k)))|\geq k^2C\right\}.
\end{equation*}
Then Markov's inequality and (\ref{t31}) give
\begin{equation*}
\mathbf{P}[\mathcal {F}_k]\leq 1/k^2.
\end{equation*}
Since $\sum_{k=1}^{\infty}\mathbf{P}[\mathcal {F}_k]<\infty$, the
Borel-Cantelli lemma implies that almost surely only finitely many
$\mathcal {F}_k$'s occur.  Then, by (\ref{e84}) and the definition
of $p_k$ we obtain
\begin{equation}\label{e16}
\frac{T_{p_k}(0,\mathcal
{C}_{\infty}(p_k))}{-\frac{4}{3}\log(p_k-1/2)}\rightarrow
\frac{1}{2\sqrt{3}\pi}~~\mbox{ a.s. as }k\rightarrow\infty.
\end{equation}
Observe that for $p_{k+1}\leq p\leq p_k$,
\begin{equation*}
\frac{T_{p_k}(0,\mathcal {C}_{\infty}(p_k))}{-\log(p_{k+1}-1/2)}\leq
\frac{T_p(0,\mathcal {C}_{\infty}(p))}{-\log(p-1/2)}\leq
\frac{T_{p_{k+1}}(0,\mathcal {C}_{\infty}(p_{k+1}))}{-\log(p_k-1/2)}
\end{equation*}
and
\begin{equation*}
\frac{\log(p_k-1/2)}{\log(p_{k+1}-1/2)}=\frac{k^4}{(k+1)^4}\rightarrow
1~~\mbox{ as }k\rightarrow\infty.
\end{equation*}
Then we derive (\ref{t32}) from (\ref{e16}) easily.

Combining (\ref{t42}), (\ref{t31}) and (\ref{e14}) gives
(\ref{t33}).

Combining (\ref{t43}), (\ref{t35}) and (\ref{e14}) gives
(\ref{t34}).

By (\ref{t46}) and (\ref{e14}), there exists a function $\eta(p)$
with $\eta(p)\rightarrow 0$ as $p\searrow 1/2$, such that
\begin{equation}\label{e83}
\frac{T_{1/2}(0,\partial
B(L(p)))+(1+\eta(p))\frac{2}{3\sqrt{3}\pi}\log(p-1/2)}{\sqrt{\left(\frac{1}{2\pi^2}-\frac{2}{3\sqrt{3}\pi}\right)\frac{4}{3}\log
(p-1/2)}}\stackrel{d}\longrightarrow N(0,1)\quad\mbox{as $p\searrow
1/2$}.
\end{equation}
Using Markov's inequality and (\ref{t31}), we get that for all
$p>1/2$,
\begin{equation*}
\mathbf{P}[|T_p(0,\mathcal {C}_{\infty}(p))-T_{1/2}(0,\partial
B(L(p)))|\geq C(-\log(p-1/2))^{1/3}]\leq
\frac{1}{(-\log(p-1/2))^{1/3}}.
\end{equation*}
This inequality together with (\ref{e83}) gives (\ref{t36}).
\end{proof}

\section{Subsequential limits for critical FPP}\label{subseq}
In this section, we give the proof of Theorem \ref{t1}.  As
mentioned earlier, we will use large deviation estimates for CLE$_6$
loops derived in \cite{MWW16}.  Recall that $\mathbb{P}$ is the
probability measure of CLE$_6$ in $\overline{\mathbb{D}}$.  In the
following, we write $\mathbf{P}$ for $\mathbf{P}_{1/2}$.
\subsection{CLE$_6$ nesting estimates}
Let us define some notation before stating the result of
\cite{MWW16}.

If $D$ is a simply connected planar domain with $0\in D$, the
\textbf{conformal radius} of $D$ viewed from 0 is defined to be
$\CR(D):=|g'(0)|^{-1}$, where $g$ is any conformal map from $D$ to
$\mathbb{D}$ that sends 0 to 0.

For $k\in\mathbb{N}$, let $\mathcal {L}_k$ be the $k$th largest
CLE$_6$ loop that surrounds 0 in $\overline{\mathbb{D}}$, and let
$U_k$ be the connected component of the open set
$\mathbb{D}\backslash\mathcal {L}_k$ that contains 0.  Write
$U_0:=\mathbb{D}$.  For $k\in \mathbb{N}$, define
$$Z_k:=\log\CR(U_{k-1})-\log\CR(U_k).$$
Proposition 1 in \cite{SSW09} says that $\{Z_k\}_{k\in\mathbb{N}}$
are i.i.d. random variables.  Furthermore, the log moment generating
function of $Z_1$ was computed in \cite{SSW09} and is given by
\begin{equation*}
\Lambda(\lambda):=\log \mathbb{E}[\exp(\lambda
B_k)]=\log\left(\frac{1}{2\cos(\pi\sqrt{1/9+4\lambda/3})}\right),\quad\mbox{for
$-\infty<\lambda<5/48$.}
\end{equation*}
Define the Fenchel-Legendre transform $\Lambda^{\star}:
\mathbb{R}\rightarrow [0,\infty]$ of $\Lambda$ by
\begin{equation*}
\Lambda^{\star}(x):=\sup_{\lambda\in \mathbb{R}}\{\lambda
x-\Lambda(\lambda)\}.
\end{equation*}
Write
\begin{equation*}
\gamma(\nu):=\left\{
\begin{aligned}
&\nu\Lambda^{\star}(1/\nu),&\mbox{ if }\nu>0,\\
&5/48,&\mbox{if }\nu=0.
\end{aligned}
\right.
\end{equation*}
We denote by $\nu_1$ the unique value of $\nu\geq 0$ such that
$\gamma(\nu)=1$.

The following lemma for the nesting of CLE$_6$ loops was proved in
\cite{MWW16}.
\begin{lemma}[Lemma 4.3 in \cite{MWW16}]\label{l1}
Let $\Gamma$ be a CLE$_6$ in $\overline{\mathbb{D}}$, and fix $
\nu\geq 0$.  Then for all fixed sufficiently large constant $M>1$,
and for all functions $\varepsilon\mapsto \delta(\varepsilon)$
decreasing to 0 sufficiently slowly as $\varepsilon\rightarrow 0$,
the event that:
\begin{itemize}
\item[(i)] there is a loop which is contained in the annulus
$\overline{\mathbb{D}(\varepsilon)}\backslash
\mathbb{D}(\varepsilon/M)$ and which surrounds 0, and
\item[(ii)] the index $J$ of the outermost such loop in the annulus
$\overline{\mathbb{D}(\varepsilon)}\backslash
\mathbb{D}(\varepsilon/M)$ satisfies $\nu\log(1/\varepsilon)\leq
J\leq (\nu+\delta(\varepsilon))\log(1/\varepsilon)$,
\end{itemize}
has probability at least $\varepsilon^{\gamma(\nu)+o(1)}$ as
$\varepsilon\rightarrow 0$.
\end{lemma}
However, we cannot use Lemma \ref{l1} directly.  We need to modify
it slightly:
\begin{lemma}\label{l18}
Let $\Gamma$ be a CLE$_6$ in the unit disk $\overline{\mathbb{D}}$,
and fix $\nu\geq 0$.  Then for all fixed sufficiently large constant
$M>1$, and for all functions $\varepsilon\mapsto
\delta(\varepsilon)$ decreasing to 0 sufficiently slowly as
$\varepsilon\rightarrow 0$, the event $\mathcal
{E}(\varepsilon,\delta,\nu)$ that:
\begin{itemize}
\item[(i)] $\mathcal {L}_1\subset\mathbb{D}\backslash \mathbb{D}(1/2)$, and
\item[(ii)] there exists a loop $\mathcal
{L}_J\subset\overline{\mathbb{D}(\varepsilon)}\backslash
\mathbb{D}(\varepsilon/M)$, with the the index $J$ satisfying
$\nu\log(1/\varepsilon)\leq J\leq
(\nu+\delta(\varepsilon))\log(1/\varepsilon)$,
\end{itemize}
has probability at least $\varepsilon^{\gamma(\nu)+o(1)}$ as
$\varepsilon\rightarrow 0$.
\end{lemma}
Note that condition (ii) of Lemma \ref{l18} is similar as the
conditions of Lemma \ref{l1}, except that it does not require
$\mathcal {L}_J$ be the outermost loop in
$\overline{\mathbb{D}(\varepsilon)}\backslash
\mathbb{D}(\varepsilon/M)$.  Before we prove Lemma \ref{l18}, let us
state a standard fact of complex analysis that will be used in the
proof.
\begin{lemma}[see e.g. Corollary 3.25 in \cite{Law05}]\label{l19}
Let $D,D'$ be two Jordan domains.  If $f: D\rightarrow D'$ is a
conformal transformation, then for all $w\in D$, $0<r<1$ and all
$|z-w|\leq r \dist(w,\partial D)$,
\begin{equation*}
|f(z)-f(w)|\leq \frac{4|z-w|}{(1-r)^2}\frac{\dist(f(w),\partial
D')}{\dist(w,\partial D)}.
\end{equation*}
\end{lemma}

\begin{proof}[Proof of Lemma \ref{l18}]
Theorem \ref{t5} together with RSW and FKG implies that, there
exists $p_0\in (0,1)$ such that
\begin{equation}\label{e85}
\mathbb{P}[\mathcal {L}_1\subset \mathbb{D}\backslash
\mathbb{D}(1/2)]\geq p_0.
\end{equation}
Suppose that the event $\{\mathcal {L}_1\subset \mathbb{D}\backslash
\mathbb{D}(1/2)\}$ holds.  Let $f:\overline{U}_1\rightarrow
\overline{\mathbb{D}}$ be a continuous function that maps $U_1$
conformally onto $\mathbb{D}$ with $f(0)=0$.  By Lemma \ref{l19} and
Schwarz Lemma, for all $|z|<1/2$,
\begin{equation*}
|z|\leq |f(z)|\leq \frac{8|z|}{(1-2|z|)^2}.
\end{equation*}
Therefore, for fixed large $M$ and all small $\varepsilon$, one has
\begin{equation}\label{e51}
f\left(\overline{\mathbb{D}(\varepsilon)}\backslash
\mathbb{D}(\varepsilon/M)\right)\supset\overline{\mathbb{D}(\varepsilon)}\backslash
\mathbb{D}(10\varepsilon/M).
\end{equation}
By the conformal invariance and renewal property of of CLE$_6$ (see
e.g. Proposition 1 of \cite{SSW09}), the law of
$f(\Gamma|_{\overline{U}_1})$ is CLE$_6$ in $\overline{\mathbb{D}}$.
By Lemma \ref{l1}, for $f(\Gamma|_{\overline{U}_1})$ in
$\overline{\mathbb{D}}$, we know that for large $M$, the event
$\widetilde{\mathcal {E}}(\varepsilon,\delta,\nu)$ that there is a
loop which is contained in the annulus
$\overline{\mathbb{D}(\varepsilon)}\backslash
\mathbb{D}(10\varepsilon/M)$ and which surrounds 0, and the index
$\tilde{J}$ of this loop satisfies $\nu\log(1/\varepsilon)\leq
\tilde{J}-1\leq (\nu+\delta(\varepsilon))\log(1/\varepsilon)$, has
probability at least $\varepsilon^{\gamma(\nu)+o(1)}$ as
$\varepsilon\rightarrow 0$.  Note that the index $J$ of the preimage
$\mathcal {L}_J$ of the above loop equals $\tilde{J}+1$, and
$\mathcal {L}_J\subset\overline{\mathbb{D}(\varepsilon)}\backslash
\mathbb{D}(\varepsilon/M)$ by (\ref{e51}).  Then by (\ref{e85}) we
have
\begin{equation*}
\mathbb{P}[\mathcal
{E}(\varepsilon,\delta,\nu)]\geq\mathbb{P}\left[\mathcal
{L}_1\subset \mathbb{D}\backslash
\mathbb{D}(1/2)\right]\mathbb{P}\left[\widetilde{\mathcal
{E}}(\varepsilon,\delta,\nu)|\mathcal {L}_1\subset
\mathbb{D}\backslash \mathbb{D}(1/2)\right]\geq
\varepsilon^{\gamma(\nu)+o(1)},
\end{equation*}
which proves Lemma \ref{l18}.
\end{proof}
\subsection{Estimates for cluster boundary loops and circuits}
Let us consider cluster boundary loops in $B(R)$ with monochromatic
boundary condition.  For $k\in\mathbb{N}$, let $\mathcal {L}_k(R)$
be the $k$th largest cluster boundary loop that surrounds 0 in
$B(R)$.  In the following, we let $M=M(\nu)$ and
$\delta(\varepsilon)$ be some fixed sufficiently large constant and
some fixed function in Lemma \ref{l18}, respectively.  Define the
discrete version of the event $\mathcal {E}(\varepsilon,\delta,\nu)$
as follows.
\begin{equation*}
\mathcal {E}(R,\varepsilon,\delta,\nu):=\left\{
\begin{aligned}
&\mbox{(i) $\mathcal {L}_1(R)\subset A(R/2,R)$ and $\mathcal {L}_1(R)$ does not touch $\partial B(R)$, and}\\
&\mbox{(ii) there exists a loop
$\mathcal {L}_J(R)\subset A(\varepsilon R/M,\varepsilon R)$, with the the index $J$}\\
&\mbox{satisfing $\nu\log(1/\varepsilon)\leq J\leq
(\nu+\delta(\varepsilon))\log(1/\varepsilon)$}
\end{aligned}
\right\}.
\end{equation*}
Similarly to Lemma \ref{l18}, for the discrete model we have the
following lemma.
\begin{lemma}\label{l2}
Fix $\nu\geq 0$.  For each $\eta>0$, there exists
$\varepsilon_0=\varepsilon_0(\eta)$, such that for each
$\varepsilon\in (0,\varepsilon_0]$, there exists
$R_0=R_0(\eta,\varepsilon,\delta)$, such that for all $R>R_0$,
\begin{equation*}
\mathbf{P}[\mathcal
{E}(R,\varepsilon,\delta,\nu)]\geq\varepsilon^{\gamma(\nu)+\eta/2},
\end{equation*}
where $\delta(\varepsilon)\rightarrow 0$ as $\varepsilon\rightarrow
0$.
\end{lemma}
\begin{proof}
The proof is standard, and is very similar to that of Proposition
3.1 in \cite{Yao18}.  For the reader's convenience we give some
details of the proof here.

By Lemma \ref{l18}, for each $\eta>0$, there exists
$\varepsilon_0=\varepsilon_0(\eta)$, such that for each
$\varepsilon\in (0,\varepsilon_0]$,
\begin{equation}\label{e58}
\mathbb{P}[\mathcal
{E}(\varepsilon,\delta,\nu)]\geq\varepsilon^{\gamma(\nu)+\eta/3},
\end{equation}
where $\delta\rightarrow 0$ as $\varepsilon\rightarrow 0$.

Let
\begin{align*}
&\mathcal {F}(R):=\mbox{the collection of cluster boundary loops
surrounding 0 in $A(R\varepsilon/M,R)$}\\
&\qquad\qquad\mbox{scaled by $1/R$},\\
&\mathcal {F}_1(R):=\mbox{the collection of cluster boundary loops
surrounding 0 in $A(R/2,R)$}\\
&\qquad\qquad\mbox{scaled by $1/R$},\\
&\mathcal {F}_2(R):=\mbox{the collection of cluster boundary loops
surrounding 0 in $A(R\varepsilon/M,R\varepsilon)$}\\
&\qquad\qquad\mbox{scaled by $1/R$},\\
&\mathcal {F}:=\mbox{the collection of CLE$_6$ loops surrounding 0
in $\overline{\mathbb{D}}\backslash\mathbb{D}(\varepsilon/M)$},\\
&\mathcal {F}_1:=\mbox{the collection of CLE$_6$ loops surrounding 0
in $\overline{\mathbb{D}}\backslash\mathbb{D}(1/2)$},\\
&\mathcal {F}_2:=\mbox{the collection of CLE$_6$ loops surrounding 0
in
$\overline{\mathbb{D}(\varepsilon)}\backslash\mathbb{D}(\varepsilon/M)$}.
\end{align*}
For $0<\varepsilon'<\varepsilon/M$, define the event
\begin{equation*}
\mathcal {A}(R,\varepsilon,\varepsilon'):=\left\{
\begin{aligned}
&\mbox{$\exists \mathcal {L}$ surrounding 0 in
$A((1-\varepsilon')R\varepsilon/M,R)$, such that}\\
&\mathcal {L}\cap
\mathbb{D}((1+\varepsilon')R\varepsilon/M)\backslash
\mathbb{D}((1-\varepsilon')R\varepsilon/M)\neq\emptyset
\end{aligned}
\right\}.
\end{equation*}
Assume that $\mathcal {A}(R,\varepsilon,\varepsilon')$ holds and $R$
is large enough (depending on $\varepsilon'$).  Then we have a
polychromatic 3-arm event from a ball of radius
$3\varepsilon'R\varepsilon/M$ centered at a point $z\in\partial
\mathbb{D}((1-\varepsilon')R\varepsilon/M)$ to a distance of order
$R\varepsilon/M$ in $A((1-\varepsilon')R\varepsilon/M,R)$.  For a
fixed $z\in\partial \mathbb{D}((1-\varepsilon')R\varepsilon/M)$, the
corresponding 3-arm event happens with probability at most
$O((\varepsilon')^2)$ (see e.g. Lemma 6.8 in \cite{Sun11}).  From
this one easily obtains $\mathbf{P}[\mathcal
{A}(R,\varepsilon,\varepsilon')]\leq O(\varepsilon')$.  Then Theorem
\ref{t5} implies that $\mathcal {F}(R)$ converges in distribution to
$\mathcal {F}$ as $R\rightarrow\infty$.  Because of the choice of
topology, we can find coupled versions of $\mathcal {F}_R$ and
$\mathcal {F}$ on the same probability space such that
$\dist(\mathcal {F}(R),\mathcal {F})\rightarrow 0$ almost surely as
$R\rightarrow \infty$.  Similarly to the above argument, in the
above coupling we have $\dist(\mathcal {F}_1(R),\mathcal
{F}_1)\rightarrow 0$ and $\dist(\mathcal {F}_2(R),\mathcal
{F}_2)\rightarrow 0$ in probability as $R\rightarrow \infty$.

For $0<\varepsilon'<\varepsilon/M$, define the event
\begin{equation*}
\mathcal {B}(R,\varepsilon,\varepsilon'):=\{\exists \mathcal
{L},\mathcal {L}'\in \mathcal {F}(R)\mbox{ such that
}\textrm{d}(\mathcal {L},\mathcal {L}')<\varepsilon'\}.
\end{equation*}
Similarly to the proof of Proposition 3.1 in \cite{Yao18}, by using
that polychromatic half-plane 3-arm exponent is 2 and the
polychromatic plane 6-arm exponent is larger that 2 (see e.g.
\cite{Nol08}), one can prove that $\mathbf{P}[\mathcal
{B}(R,\varepsilon,\varepsilon')]\rightarrow 0$ as
$\varepsilon'\rightarrow 0$ and all large $R$ (depending on
$\varepsilon'$).  This implies that in the above coupling, for all
$1\leq i\leq (\nu+\delta(\varepsilon))\log(1/\varepsilon)$ with
$\mathcal {L}_i\subset \overline{\mathbb{D}(\varepsilon)}\backslash
\mathbb{D}(\varepsilon/M)$, $\textrm{d}(\mathcal {L}_i,(1/R)\mathcal
{L}_i(R))\rightarrow 0$ in probability as $R\rightarrow \infty$ (if
such $i$ exists).  This fact combined with the above argument gives
the the desired result.
\end{proof}

Now let us consider circuits in $B(R)$.  Let $D_0:=B(R)$.  Take
$\mathcal {C}_1(R)$ the outermost yellow circuit surrounding 0 in
$D_0$ and let $D_1$ be the component of $D_0\backslash \mathcal
{C}_1(R)$ that contains 0, then take $\mathcal {C}_2(R)$ the
outermost yellow circuit surrounding 0 in $D_1$ and let $D_2$ be the
component of $D_1\backslash \mathcal {C}_2(R)$ that contains 0, and
so on.  We call $\mathcal {C}_k(R)$ the $k$th outermost yellow
circuit that surrounds 0 in $B(R)$.  Define the event
\begin{equation*}
\widehat{\mathcal {E}}(R,\varepsilon,\delta,\nu):=\left\{
\begin{aligned}
&\mbox{(i) $\mathcal {L}_1(R)\subset A(R/2,R)$ and $\mathcal {L}_1(R)$ does not touch $\partial B(R)$, and}\\
&\mbox{(ii) there exists a circuit
$\mathcal {C}_J(R)\subset A(\varepsilon R/M,\varepsilon R)$, with the the index $J$}\\
&\mbox{satisfing $\nu\log(1/\varepsilon)\leq J\leq
(\nu+\delta(\varepsilon))\log(1/\varepsilon)$}
\end{aligned}
\right\}.
\end{equation*}
We will derive the following lemma from Lemma \ref{l2} by a ``color
switching trick".

\begin{lemma}\label{ll2}
Fix $\nu\geq 0$.  For each $\eta>0$, there exists
$\varepsilon_0=\varepsilon_0(\eta)$, such that for each
$\varepsilon\in (0,\varepsilon_0]$, there exists
$R_0=R_0(\eta,\varepsilon,\delta)$, such that for all $R>R_0$,
\begin{equation*}
\mathbf{P}[\widehat{\mathcal
{E}}(R,\varepsilon,\delta,\nu)]\geq\varepsilon^{\gamma(\nu)+\eta/2},
\end{equation*}
where $\delta(\varepsilon)\rightarrow 0$ as $\varepsilon\rightarrow
0$.
\end{lemma}
\begin{proof}
We use the proof of the second item in Proposition 2.4 of
\cite{Yao18}.  Assume that $B(R)$ has monochromatic boundary
condition.  It is proved that, for any fixed $n\in \mathbb{N}$, one
can construct a bijection $f_n$ between the sets
$\{\omega:\rho(r,R)=n\}$ and $\{\omega':N(r,R)=n\}$ as follows.
Given a configuration $\omega\in\{\rho(r,R)=n\}$, if $n$ is odd, we
switch the colors of the sites in $D_1\backslash
D_2,\ldots,D_{n-2}\backslash D_{n-1},D_n$; if $n$ is even, we switch
the colors of the sites in $D_1\backslash
D_2,\ldots,D_{n-1}\backslash D_n$.  Observe that under the
transformation $f_n$, each $\mathcal {C}_k(R)$ in the original
configuration $\omega$ corresponds to the outermost monochromatic
circuit inside $\mathcal {L}_k(R)$ that surrounds 0 in the
configuration $f_n(\omega)$.  This observation together with Lemma
\ref{l2} implies Lemma \ref{ll2}.
\end{proof}

Note that $\widehat{\mathcal {E}}(R,\varepsilon,\delta,\nu)$ is an
event about circuits surrounding 0.  Similarly, one can define the
event $\widehat{\mathcal {E}}(n;R,\varepsilon,\delta,\nu)$ about
circuits surrounding the point $n$, which includes analogous
conditions (i) and (ii).  It is clear that $\widehat{\mathcal
{E}}(0;R,\varepsilon,\delta,\nu)=\widehat{\mathcal
{E}}(R,\varepsilon,\delta,\nu)$.  For $k\in \mathbb{N}$, we write
$\mathcal {E}(n;k)=\widehat{\mathcal
{E}}(n;(\varepsilon/M)^{-k},\varepsilon,\delta,\nu)$ for notational
convenience.

Assume that the event $\mathcal {E}(n;k)$ holds.  Let $\mathcal
{\mathcal {C}}^+(n;k)$ be the outermost yellow circuit in
$B(n,(\varepsilon/M)^{-k})$ that surrounds $n$, and let $\mathcal
{\mathcal {C}}^-(n;k)$ be the outermost circuit satisfying condition
(ii).  Assume that $k\geq 2$ and the event $\mathcal
{E}(n;k)\cap\mathcal {E}(n;k-1)$ holds. Then define the event
\begin{equation*}
\mathcal {G}(n;k):=\{\mbox{there exists a blue path touching both
$\mathcal {C}^-(n;k)$ and $\mathcal {C}^+(n;k-1)$}\}.
\end{equation*}
For $i\geq 1$ and $j\geq i+1$, define the events
\begin{align*}
\mathcal {F}(n;i,j)=\mathcal
{F}(n;i,j;\varepsilon,\delta,\nu):=\bigcap_{k=i}^{j}\mathcal
{E}(n;k)\cap \bigcap_{k=i+1}^{j}\mathcal {G}(n;k).
\end{align*}
For convenience, write $\mathcal {F}(n;i,i):=\mathcal {E}(n;i)$ and
$\mathcal {F}(n;i):=\mathcal {F}(n;1,i)$.

\begin{lemma}\label{l20}
Fix $\nu\geq 0$.  For each $\eta>0$, there exists
$\varepsilon_0=\varepsilon_0(\eta)$, such that for each
$\varepsilon\in (0,\varepsilon_0]$, there exists
$C(\eta,\varepsilon,\delta,\nu)>0$, such that for all $j\geq i\geq
1$,
\begin{equation*}
\mathbf{P}[\mathcal {F}(0;i,j)]\geq
C\varepsilon^{(j-i+1)(\gamma(\nu)+\eta)},
\end{equation*}
where $\delta\rightarrow 0$ as $\varepsilon\rightarrow 0$.
\end{lemma}
\begin{proof}
It is obvious that the case of $j=i\geq 1$ follows from Lemma
\ref{ll2}.  So we assume that $j>i\geq 1$ in the following.

It is clear that conditioned on the event $\mathcal
{E}(0;i)\cap\cdots\cap \mathcal {E}(0;j)$ and the circuits $\mathcal
{C}^-(0;k)$, $\mathcal {C}^+(0;k), i\leq k\leq j$, the events
$\mathcal {G}(0;k)$, $i+1\leq k\leq j$ are independent, and
furthermore, the probability measure of the configuration in
$D_k:=\{\mbox{interior of }\mathcal {C}^-(0;k)\}\backslash
\overline{\mathcal {C}^+(0;k-1)}$ is just the percolation measure in
$D_k$ conditioned that there is a blue circuit surrounding 0 in
$B((M/\varepsilon)^{k-1})\backslash \overline{\mathcal
{C}^+(0;k-1)}$ and for each site in $\mathcal {C}^+(0;k-1)$ there is
a blue path touching this site and $\partial
B((M/\varepsilon)^{k-1})$.    Then by FKG and RSW, there exists an
absolute constant $p_0\in (0,1)$ (depending only on $M$), such that
\begin{align}
&\mathbf{P}[\mathcal {F}(0;i,j)]\nonumber\\
&=\sum_{\{C^-(0;k),C^+(0;k): i\leq k\leq j\}}\mathbf{P}[\mathcal
{C}^-(0;k)=C^-(0;k),\mathcal {C}^+(0;k)=C^+(0;k),i\leq k\leq
j]\nonumber\\
&\quad\times\mathbf{P}\left[\bigcap_{k=i+1}^{j}\mathcal
{G}(0;k)\mid\mathcal
{C}^-(0;k)=C^-(0;k),\mathcal {C}^+(0;k)=C^+(0;k),i\leq k\leq j\right]\nonumber\\
&=\sum_{\{C^-(0;k),C^+(0;k): i\leq k\leq j\}}\mathbf{P}[\mathcal
{C}^-(0;k)=C^-(0;k),\mathcal {C}^+(0;k)=C^+(0;k),i\leq k\leq
j]\nonumber\\
&\quad\times\prod_{k=i+1}^{j}\mathbf{P}\left[\mathcal
{G}(0;k)\mid\mathcal
{C}^-(0;k)=C^-(0;k),\mathcal {C}^+(0;k)=C^+(0;k),i\leq k\leq j\right]\nonumber\\
&\geq\prod_{k=i}^{j}\mathbf{P}[\mathcal
{E}(0;k)]\prod_{k=i}^{j-1}\mathbf{P}[\mathcal
{A}_1((M/\varepsilon)^k/2,(M/\varepsilon)^kM)]\nonumber\\
&\geq p_0^{j-i}\prod_{k=i}^{j}\mathbf{P}[\mathcal
{E}(0;k)]\label{e86}.
\end{align}
By Lemma \ref{ll2}, for each $\eta>0$, there exists
$\varepsilon_1=\varepsilon_1(\eta)$, such that for each
$\varepsilon\in (0,\varepsilon_1]$, there exists
$k_1=k_1(\eta,\varepsilon,\delta)$, such that for all $k\geq k_1$,
\begin{equation*}
\mathbf{P}[\mathcal {E}(0;k)]\geq\varepsilon^{\gamma(\nu)+\eta/2},
\end{equation*}
where $\delta\rightarrow 0$ as $\varepsilon\rightarrow 0$.  This and
(\ref{e86}) implies that for each $\eta>0$, there exists
$\varepsilon_2=\varepsilon_2(\eta)$, such that for each
$\varepsilon\in (0,\varepsilon_2]$ and for all $j>i\geq k_1$,
\begin{equation*}
\mathbf{P}[\mathcal {F}(0;i,j)]\geq p_0^{j-i}
\varepsilon^{(j-i+1)(\gamma(\nu)+\eta/2)}\geq
\varepsilon^{(j-i+1)(\gamma(\nu)+\eta)},
\end{equation*}
which gives the lemma immediately.
\end{proof}
The following result is easy to derive from Lemma \ref{l20}, and is
needed for our second moment method.
\begin{lemma}\label{l3}
Fix $\nu\geq 0$.  For each $\eta>0$, there exists
$\varepsilon_0=\varepsilon_0(\eta)$, such that for each
$\varepsilon\in (0,\varepsilon_0]$, there exists
$C(\eta,\varepsilon,\delta,\nu)>0$, such that for all $n\geq 1$ and
$j\geq\max\{1,\lfloor\log_{M/\varepsilon}(n/2)\rfloor\}$,
\begin{equation*}
\mathbf{P}[\mathcal {F}(0;j)\cap \mathcal {F}(n;j)]\leq
Cn^{-(\gamma(\nu)+\eta)}(M/\varepsilon)^
{j(\gamma(\nu)+\eta)}\left\{\mathbf{P}[\mathcal {F}(0;j)]\right\}^2,
\end{equation*}
where $\delta\rightarrow 0$ as $\varepsilon\rightarrow 0$.
\end{lemma}
\begin{proof}
By Lemma \ref{l20}, for each $\eta>0$, we can choose a small
$\varepsilon_0=\varepsilon_0(\eta)>0$, such that for each
$\varepsilon\in (0,\varepsilon_0]$, there exists
$C(\eta,\varepsilon,\delta,\nu)>0$, such that for all $j\geq i\geq
1$,
\begin{equation}\label{e87}
\mathbf{P}[\mathcal {F}(0;i,j)]\geq
C\varepsilon^{(j-i+1)(\gamma(\nu)+\eta)},
\end{equation}
where $\delta\rightarrow 0$ as $\varepsilon\rightarrow 0$.
Therefore, if $1\leq n<2M/\varepsilon$, then we have
\begin{equation}\label{e88}
\mathbf{P}[\mathcal {F}(0;j)\cap \mathcal {F}(n;j)]\leq
\mathbf{P}[\mathcal {F}(0;j)]\leq C^{-1}\varepsilon^
{-j(\gamma(\nu)+\eta)}\left\{\mathbf{P}[\mathcal
{F}(0;j)]\right\}^2.
\end{equation}
In the following, we assume that $n\geq 2M/\varepsilon$.   Write
\begin{equation*}
j_1:=\lfloor\log_{M/\varepsilon}(n/2)\rfloor,~j_2:=\lceil\log_{M/\varepsilon}(3n/2)\rceil.
\end{equation*}
Note that $j_2$ equals $j_1+1$ or $j_1+2$.   By (\ref{e87}) and the
argument in the proof of Lemma \ref{l20}, there exists
$C_1(\eta,\varepsilon,\delta,\nu)>0$, such that
\begin{equation}\label{e56}
\mathbf{P}[\mathcal {F}(0;j)]\geq p_0\mathbf{P}[\mathcal
{F}(0;j_1)]\mathbf{P}[\mathcal {F}(0;j_1+1,j)]\geq
C_1\varepsilon^{(j-j_1)(\gamma(\nu)+\eta)}\mathbf{P}[\mathcal
{F}(0;j_1)],
\end{equation}
where we let $\mathbf{P}[\mathcal {F}(0;j_1+1,j)]=1$ if $j=j_1$.  If
$j_2>j$, the above inequality gives
\begin{equation}\label{e89}
\mathbf{P}[\mathcal {F}(0;j)\cap \mathcal {F}(n;j)]\leq
(\mathbf{P}[\mathcal {F}(0;j_1)])^2\leq C_1^{-1}\varepsilon^
{-2(\gamma(\nu)+\eta)}\left\{\mathbf{P}[\mathcal
{F}(0;j)]\right\}^2.
\end{equation}
In the following, we assume that $j_2\leq j$.  By (\ref{e87}) and
the argument in the proof of Lemma \ref{l20}, there exists
$C_2(\eta,\varepsilon,\delta,\nu)>0$, such that
\begin{align}
\mathbf{P}[\mathcal {F}(0;j)]&\geq p_0^2\mathbf{P}[\mathcal
{F}(0;j_1)]\mathbf{P}[\mathcal {F}(0;j_1+1,j_2)]\mathbf{P}[\mathcal
{F}(0;j_2,j)]\nonumber\\
&\geq C_2\mathbf{P}[\mathcal {F}(0;j_1)]\mathbf{P}[\mathcal
{F}(0;j_2,j)].\label{e57}
\end{align}
Using (\ref{e56}) and (\ref{e57}), there is a
$C_3(\eta,\varepsilon,\delta,\nu)>0$ such that
\begin{align}
\mathbf{P}[\mathcal {F}(0;j)\cap \mathcal {F}(n;j)]&\leq
(\mathbf{P}[\mathcal {F}(0;j_1)])^2\mathbf{P}[\mathcal
{F}(0;j_2,j)]\nonumber\\
&\leq
C_1^{-1}C_2^{-1}(M/\varepsilon)^{(j-j_1)(\gamma(\nu)+\eta)}\left\{\mathbf{P}[\mathcal
{F}(0;j)]\right\}^2\nonumber\\
&\leq C_3n^{-(\gamma(\nu)+\eta/2)}(M/\varepsilon)^
{j(\gamma(\nu)+\eta)}\left\{\mathbf{P}[\mathcal
{F}(0;j)]\right\}^2.\label{e55}
\end{align}
Combining (\ref{e88}), (\ref{e89}) and (\ref{e55}) gives the desired
result.
\end{proof}
\subsection{Proof of Theorem \ref{t1}} We will use the second moment
method to prove Theorem \ref{t1}.  The following lemma is a key
ingredient.  For $j\in\mathbb{N}$, write
$j^*=\lfloor\log_{M/\varepsilon}(2^{j-1})\rfloor$.  If $j^*\geq 1$,
let
\begin{equation*}
\quad X_j:=\sum_{n\in[2^j,2^{j+1})}\mathbf{1}_{\mathcal {F}(n;j^*)}.
\end{equation*}
\begin{lemma}\label{l17}
Fix $\nu\in [0,\nu_1)$.  For each $\eta\in(0,1-\gamma(\nu))$, there
exists $\varepsilon_0=\varepsilon_0(\eta)$, such that for each
$\varepsilon\in (0,\varepsilon_0]$, there exists
$C(\eta,\varepsilon,\delta,\nu)>0$, such that for all $j$ with
$j^*\geq 1$,
\begin{equation*}
\mathbf{P}[X_j\geq 1]\geq C,
\end{equation*}
where $\delta\rightarrow 0$ as $\varepsilon\rightarrow 0$.
\end{lemma}
\begin{proof}
It is clear that
\begin{equation*}
\mathbf{E}[X_j]=2^j\mathbf{P}[\mathcal {F}(0;j^*)].
\end{equation*}
By Lemmas \ref{l20} and \ref{l3}, for each
$\eta\in(0,1-\gamma(\nu))$, there exists
$\varepsilon_0=\varepsilon_0(\eta)$, such that for each
$\varepsilon\in (0,\varepsilon_0]$, there exist $C_1,C_2,C>0$ that
depend on $\eta,\varepsilon,\delta,\nu$, such that for all $j$ with
$j^*\geq 1$,
\begin{align*}
\mathbf{E}[X_j^2]&=\sum_{2^j\leq m<n<2^{j+1}}2\mathbf{P}[\mathcal
{F}(m;j^*)\cap \mathcal {F}(n;j^*)]+\sum_{2^j\leq
n<2^{j+1}}\mathbf{P}[\mathcal {F}(n;j^*)]\\
&\leq C_12^{(\gamma(\nu)+\eta)j}(\mathbf{P}[\mathcal
{F}(0;j^*)])^2\sum_{2^j\leq
m<n<2^{j+1}}(n-m)^{-(\gamma(\nu)+\eta)}\\
&\quad+C_22^j2^{(\gamma(\nu)+\eta)j}(\mathbf{P}[\mathcal {F}(0;j^*)])^2\\
&\leq C4^j(\mathbf{P}[\mathcal {F}(0;j^*)])^2.
\end{align*}
Then we have
\begin{equation*}
\mathbf{P}[X_j\geq 1]\geq
\frac{(\mathbf{E}[X_j])^2}{\mathbf{E}[X_j^2]}\geq C.
\end{equation*}
\end{proof}
\begin{proof}[Proof of Theorem \ref{t1}]
Fix any $\nu\in [0,\nu_1)$.  By Lemma \ref{l17}, for each
$\eta\in(0,1-\gamma(\nu))$, there exists
$\varepsilon_0=\varepsilon_0(\eta)$, such that for each
$\varepsilon\in (0,\varepsilon_0]$, there exists
$C(\eta,\varepsilon,\delta,\nu)>0$, such that for all $j$ with
$j^*\geq 2$,  with probability at least $C$ there exists $2^j\leq
n<2^{j+1}$, such that $\mathcal {F}(n;j^*)$ occurs.  Suppose that
$\mathcal {F}(n;j^*)$ holds.  For each $2\leq k\leq j^*$, since the
outermost cluster boundary loop surrounding $n$ in
$B(n,(M/\varepsilon)^{k-1})$ does not touch $\partial
B(n,(M/\varepsilon)^{k-1})$ and there exists a blue path touching
both $\mathcal {C}^-(n;k)$ and $\mathcal {C}^+(n;k-1)$, we know that
$\mathcal {C}^+(n;k-1)$ is the outermost yellow circuit surrounding
$n$ inside $\mathcal {C}^-(n;k)$.   Therefore, all the $i$th's
outermost yellow circuits surrounding $n$ in
$B(n,(M/\varepsilon)^{j^*})$ from outside to inside are $\mathcal
{C}^+(n;j^{*}),\ldots,\mathcal {C}^-(n;j^{*}),\mathcal
{C}^+(n;j^{*}-1),\ldots,\mathcal {C}^-(n;2),\mathcal
{C}^+(n;1),\ldots,\mathcal {C}^-(n;1),\ldots$, where $\mathcal
{C}^-(n;1)\subset A(1,M)$, and the number of circuits in the
subsequence $\mathcal {C}^+(n;k),\ldots,\mathcal {C}^-(n;k)$ is in
$[\nu\log(1/\varepsilon),(\nu+\delta)\log(1/\varepsilon)]$ for each
$1\leq k\leq j^*$.   Then by item (i) of Proposition \ref{p3}, we
have
\begin{equation*}
j^*\nu\log(1/\varepsilon)\leq T(n,\partial
B(n,(M/\varepsilon)^{j*}))\leq j^*(\nu+\delta)\log(1/\varepsilon)+M,
\end{equation*}
where $j^*=\lfloor\log_{M/\varepsilon}(2^{j-1})\rfloor$ and
$\delta\rightarrow 0$ as $\varepsilon\rightarrow 0$.  Denote the
above event by $\mathcal {B}_j$.  Since the events $\{\mathcal
{B}_{3j}\}_j$ are independent and
$\sum_{j=1}^{\infty}\mathbf{P}[\mathcal {B}_{3j}]=\infty$, the
Borel-Cantelli lemma implies that almost surely infinitely many of
the events $\{\mathcal {B}_{3j}\}_j$ occur.  So almost surely there
exists a random subsequence $\{m_i:i\geq 1\}$ with $2^{j_i}\leq
m_i<2^{j_i+1}$ and $j_1<j_2<\cdots$, such that
\begin{equation}\label{e52}
j_i^*\nu\log(1/\varepsilon)\leq T(m_i,\partial
B(m_i,(M/\varepsilon)^{j_i*}))\leq
j_i^*(\nu+\delta)\log(1/\varepsilon)+M.
\end{equation}

Define
\begin{align*}
Y_j:=&\mbox{the maximal number of disjoint yellow circuits
intersecting $[0,2^{j+2}]$ with}\\
&\mbox{Euclidean diameters greater than $\varepsilon2^{j-1}/M$}.
\end{align*}
Observe that for $2^j\leq m<2^{j+1}$, the maximal number of disjoint
yellow circuits surrounding either 0 or $m$ and intersecting
$\mathbb{C}\backslash(B(2^{j-1})\cup B(m,2^{j^*}))$ is smaller than
or equal to $Y_j$.  From this fact we obtain that
\begin{equation}\label{e90}
T(0,\partial B(2^{j-1}))+T(m,\partial B(m,2^{j^*}))\leq a_{0,m}\leq
T(0,\partial B(2^{j-1}))+T(m,\partial B(m,2^{j^*}))+Y_{j}.
\end{equation}

By RSW and BK inequality, there exists $C_1>0$ depending on
$\varepsilon$ and $M$ but independent of $j$, such that
\begin{equation*}
\mathbf{P}[Y_j\geq \sqrt{j}]\leq\exp(-C_1\sqrt{j}).
\end{equation*}
Therefore, $\sum_{j=1}^{\infty}\mathbf{P}[Y_j\geq \sqrt{j}]<\infty$.
Then the Borel-Cantelli lemma implies that almost surely
\begin{equation}\label{e53}
Y_j\leq \sqrt{j}~~\mbox{ for all large $j$}.
\end{equation}
By (\ref{t41}), given any fixed small $\delta>0$, almost surely
\begin{equation}\label{e54}
\frac{1}{2\sqrt{3}\pi}-\delta\leq \frac{T(0,\partial B(2^j))}{\log
2^j}\leq\frac{1}{2\sqrt{3}\pi}+\delta~~\mbox{ for all large $j$}.
\end{equation}
Combining (\ref{e52}), (\ref{e90}) (\ref{e53}) and (\ref{e54}) gives
that, for all large $m_i$, we have
\begin{align*}
&a_{0,m_i}\geq\left(\frac{1}{2\sqrt{3}\pi}-\delta\right)\log
2^{j_i-1}+j_i^*\nu\log(1/\varepsilon),\\
&a_{0,m_i}\leq\left(\frac{1}{2\sqrt{3}\pi}+\delta\right)\log
2^{j_i-1}+j_i^*(\nu+\delta)\log(1/\varepsilon)+\sqrt{j_i}+M.
\end{align*}
Therefore, by choosing $\varepsilon$ and $\delta(\varepsilon)$
sufficiently small, we know that for each fixed $\delta'>0$, almost
surely there exists a random subsequence $\{m_i:i\geq 1\}$ such that
\begin{equation*}
\left(\frac{1}{2\sqrt{3}\pi}+\nu-\delta'\right)\log m_i\leq
a_{0,m_i}\leq \left(\frac{1}{2\sqrt{3}\pi}+\nu+\delta'\right)\log
m_i.
\end{equation*}
This implies that for any fixed $\nu\in [0,\nu_1)$, almost surely
there exists a random subsequence $\{n_i:i\geq 1\}$ depending on
$\nu$, such that
\begin{equation*}
\lim_{i\rightarrow\infty}\frac{a_{0,n_i}}{\log
n_i}=\frac{1}{2\sqrt{3}\pi}+\nu.
\end{equation*}
Therefore, almost surely for all rational $\nu\in [0,\nu_1)$, there
exists a corresponding random subsequence with respect to $\nu$ as
above simultaneously, which gives the desired result for all
$\nu\in[0,\nu_1]$ immediately.
\end{proof}

\section{Cluster graph}
In this section, we give the proof of Proposition \ref{p5} and
Theorem \ref{t6}.  We write $\mathbf{P}$ for $\mathbf{P}_{1/2}$
throughout this section.

\subsection{Double circuit}
To study cluster graph, we need the notion of double circuit
introduced in \cite{WA87}, which was used by Wierman and Appel to
show that there is almost surely an infinite AB percolation cluster
on $\mathbb{T}$ for an interval of parameter values centered at
$1/2$.  A \textbf{double circuit} is a pair of disjoint circuits
$\mathcal {C},\mathcal {C}'$, such that $\mathcal {C}$ is surrounded
by $\mathcal {C}'$, and each site in $\mathcal {C}$ has a neighbor
site in $\mathcal {C}'$, and each site in $\mathcal {C}'$ has a
neighbor site in $\mathcal {C}$.  We need some additional notation.

If $W$ is a set of sites, then its internal site boundary is
\begin{equation*}
\Delta^{in} W:=\{v:v\in W\mbox{ and $v$ is adjacent to
$\mathbb{V}\backslash W$}\}.
\end{equation*}
Note that $\Delta^{in} W=\Delta(\mathbb{V}\backslash W)$.  The
exterior site boundary of $W$ is
\begin{align*}
\Delta^{\infty} W:=&\{v\in \Delta W:\mbox{there exists a path
$\gamma$ on $\mathbb{T}$ from $v$ to $\infty$ such that}\\
&\mbox{ the only site of $\gamma$ in $W\cup\Delta W$ is $v$}\}.
\end{align*}
It is well known that if $W$ is a finite and connected set, then
$\Delta^{\infty} W$ is a circuit.  The following two observations
for double circuit are elementary, and we omit the proof here.
\begin{itemize}
\item For a double circuit, there exist no sites between its two
circuits.
\item Suppose that $\mathcal {C}$ is a circuit. Then both $\Delta^{\infty}\mathcal {C}$ and $\Delta^{in}\{\mbox{interior of }\Delta^{\infty}\mathcal {C}\}$ are
circuits, and they form a double circuit.
\end{itemize}

Proposition \ref{p6} below states two combinatorial properties of
double circuit, which will be used in the proof of Proposition
\ref{p5}. The first property was essentially proved in \cite{WA87};
an analogue of the second one for AB percolation also appeared in
\cite{WA87}.  Therefore, we just sketch the proof here and omit some
details.  To state the result, we denote by
$\widetilde{\mathbb{Z}^2}$ the parallelogrammic lattice derived from
$\mathbb{T}$ by deleting the bonds parallel to the vector
$e^{2i\pi/3}$.
\begin{proposition}\label{p6}
Double circuit in $\mathbb{T}$ satisfies the following properties.
\begin{itemize}
\item[(i)] For a double circuit composed of circuits $\mathcal {C}$ and $\mathcal
{C}'$, there exists a circuit $\widetilde{\mathcal {C}}$ in
$\widetilde{\mathbb{Z}^2}$ such that $\mathcal
{C}\subset\widetilde{\mathcal {C}}\subset\mathcal {C}\cup\mathcal
{C}'$.
\item[(ii)] A cluster $\mathcal {C}$ belongs to a finite component of cluster graph
if and only if $\mathcal {C}$ is surrounded by a closed double
circuit.
\end{itemize}
\end{proposition}
\begin{proof}
We first show item (i).  Suppose that $\mathcal
{C}=(u_1,u_2,\ldots,u_m)$ and $\mathcal {C}'=(v_1,v_2,\ldots,v_n)$.
We construct the circuit $\widetilde{\mathcal {C}}$ as follows.  Let
$1\leq i\leq m$.  We claim that for any bond $(u_i,u_{i+1})$ (let
$u_{m+1}:=u_1$ if $i=m$) parallel to the vector $e^{2i\pi/3}$, there
exists a site $v_i^*$ in $\{v_1,\ldots,v_n\}$ that is adjacent to
both $u_i$ and $u_{i+1}$ by bonds parallel to vectors $1$ or
$e^{i\pi/3}$.  Suppose for a contradiction that this is not the
case.  Then three cases may occur:
\begin{enumerate}
\item neither common neighbor site of $\{u_i,u_{i+1}\}$ belongs to
$\mathcal {C}\cup\mathcal {C}'$;
\item one common neighbor site belongs to $\mathcal {C}$, while the
other does not belong to $\mathcal {C}\cup\mathcal {C}'$;
\item both the common neighbor sites belong to $\mathcal {C}$.
\end{enumerate}
One can check that each case above will lead to a contradiction with
the definition of double circuit, which gives our claim.  We insert
$v_i^*$ between $u_i$ and $u_{i+1}$ for each $(u_i,u_{i+1})$
parallel to the vector $e^{2i\pi/3}$ and replace $(u_i,u_{i+1})$
with $(u_i,v_i^*,u_{i+1})$.  Then we get the desired circuit
$\widetilde{\mathcal {C}}$, since a site can not be inserted more
than once.  If not, suppose that $v_i$ is inserted more than once.
Then there exist four sites $u_j,u_{j+1},u_k,u_{k+1}$ in $\mathcal
{C}$, such that they are adjacent to $v_i$, with the bonds
$(u_j,u_{j+1})$ and $(u_k,u_{k+1})$ parallel to vector
$e^{2i\pi/3}$.  Therefore, either $(u_j,u_{j+1})$ or $(u_k,u_{k+1})$
is surrounded by $\mathcal {C}'$.  This leads to $\mathcal
{C}\cap\mathcal {C}'\neq \emptyset$, which is a contradiction.

Now let us show item (ii).  It is obvious that if an open cluster is
surrounded by a closed double circuit, then it belongs to a finite
component of cluster graph. Conversely, it remains to show that
given any finite component of cluster graph composed of open
clusters $\mathcal {C}_1,\ldots,\mathcal {C}_n$, we can find a
closed double circuit surrounding all these clusters.  Note that
$\Delta^{\infty}\mathcal {C}_1,\ldots,\Delta^{\infty}\mathcal {C}_n$
are closed circuits. Let
$G_0=\bigcup_{i=1}^{n}\overline{\Delta^{\infty}\mathcal {C}_i}$.  It
is easy to show that $G_0$ is simply connected and has no
``bottlenecks" (cut vertices).  So $\mathcal {C}:=\Delta^{in}G_0$ is
a circuit.  It is clear that $\mathcal {C}$ is closed and $\mathcal
{C}$ surrounds $\mathcal {C}_1,\ldots,\mathcal {C}_n$.  Furthermore,
the circuit $\Delta^{\infty}\mathcal {C}$ is also closed, since each
site $v$ in $\Delta^{\infty}\mathcal {C}$ has a neighbor in
$\mathcal {C}$ and thus the graph distance between $v$ and some
$\mathcal {C}_i$ equals two.  Then the second observation just above
Proposition \ref{p6} implies that $\Delta^{\infty}\mathcal {C}$ and
$\Delta^{in}\{\mbox{interior of }\Delta^{\infty}\mathcal {C}\}$ form
a closed double circuit surrounding $\mathcal {C}_1,\ldots,\mathcal
{C}_n$.
\end{proof}

\subsection{Proof of Proposition \ref{p5}}
\begin{proof}[Proof of Proposition \ref{p5}]
For $k,n\in\mathbb{N}$, define the events
\begin{align*}
&\mathcal {F}(k):=\{\mbox{$\exists$ a closed double circuit
surrounding 0, with
Euclidean diameter larger than $k$}\},\\
&\widetilde{\mathcal {A}}(n;k):=\{\mbox{$\exists$ a closed path
connecting site $n$ to $\partial B(n,k)$ on
$\widetilde{\mathbb{Z}^2}$}\}.
\end{align*}
Note that a closed double circuit surrounding 0 must intersect some
site $n\in\mathbb{N}$.  Then by Proposition \ref{p6}, there exists a
closed circuit on $\widetilde{\mathbb{Z}^2}$ which surrounds 0 and
passes through $n$.

Denote by $\widetilde{\mathbf{P}}$ the probability measure for site
percolation on $\widetilde{\mathbb{Z}^2}$ with parameter $1/2$.  It
is well known that the critical probability of site percolation on
$\mathbb{Z}^2$, denoted by $p_c^{site}(\mathbb{Z}^2)$, is strictly
greater than that of bond percolation on $\mathbb{Z}^2$, denoted by
$p_c^{bond}(\mathbb{Z}^2)$ (see Theorem 3.28 in \cite{Gri99}).  Then
Kesten's result that $p_c^{bond}(\mathbb{Z}^2)=1/2$ (see e.g.
Theorem 11.11 in \cite{Gri99}) implies
$p_c^{site}(\mathbb{Z}^2)>1/2$.  Combining the above argument with
the exponential decay of the radius distribution result for
subcritical site percolation on $\mathbb{Z}^2$ (see e.g. Theorems 7
and 9 in \cite{BR06}), there exist $C_1,C_2>0$, such that for any
$k\geq 1$,
\begin{equation*}
\mathbf{P}[\mathcal {F}(k)]
\leq\sum_{n=1}^k\widetilde{\mathbf{P}}[\widetilde{\mathcal
{A}}(n;k)]+\sum_{n=k+1}^{\infty}\widetilde{\mathbf{P}}[\widetilde{\mathcal
{A}}(n;n)]\leq C_1k\exp(-C_2k).
\end{equation*}
From this one easily obtains that there exists $C_3>0$ such that for
any $k\in \mathbb{N}$,
\begin{equation}\label{e60}
\mathbf{P}[\mathcal {F}(k)]\leq\exp(-C_3k),
\end{equation}
which implies $\sum_{k=1}^{\infty}\mathbf{P}[\mathcal
{F}(k)]<\infty$.  By the Borel-Cantelli lemma, almost surely only
finitely many of the events $\mathcal {F}(k)$'s occur.  So almost
surely there is a random $K>0$, such that there exist no closed
double circuits surrounding $B(K)$.  Since there are infinitely many
open clusters surrounding $B(K)$ almost surely, by Proposition
\ref{p6}, they must belong to the unique infinite component
$\mathscr{C}$ of the cluster graph.

Now let us show (\ref{e91}) for cluster graph.  For
$k\in\mathbb{N}$, write
\begin{equation*}
\mathcal {E}(k):=\{\mbox{all the sites in $B(k)$ are open}\}.
\end{equation*}
Then we have
\begin{align*}
&\mathbf{P}[\mbox{there exist no closed double circuits surrounding
$B(k)$}]\\
&=\mathbf{P}[\mbox{there exist no closed double circuits surrounding
$B(k)$}|\mathcal {E}(k)]\\
&=\mathbf{P}[\mbox{the cluster
containing $B(k)$ belongs to $\mathscr{C}$}|\mathcal {E}(k)]\quad\mbox{by Proposition \ref{p6}}\\
&=\mathbf{P}[\mbox{there is a path from $\partial B(k)$ to $\infty$ without two consecutive sites being closed}|\mathcal {E}(k)]\\
&\leq \mathbf{P}[\mathscr{C}\cap B(k+2)\neq \emptyset].
\end{align*}
Combining the above inequality and (\ref{e60}), for all $k\geq 3$ we
get
\begin{align*}
\mathbf{P}[\dist(0,\mathscr{C})\geq
k]&\leq\mathbf{P}[\mbox{$\exists$ a closed double circuit
surrounding $B(k-2)$}]\\
&\leq \mathbf{P}[\mathcal {F}(k-1)]\leq \exp(-C_3(k-1)).
\end{align*}
From this we derive (\ref{e91}) for cluster graph easily.
\end{proof}
\subsection{Proof of Theorem \ref{t6}}
Before giving the proof of Theorem \ref{t6}, we need Proposition
\ref{p7} below on the geodesics of critical Bernoulli FPP.  We start
with the following definitions.

An infinite path $\gamma$ is called an \textbf{infinite geodesic} if
every subpath of $\gamma$ is a finite geodesic.  For a pair of
neighboring closed sites, if there exists an infinite geodesic from
0 to $\infty$ passing through both of them, we call them \textbf{bad
sites}.

Let us note that Proposition \ref{p7} allows us to derive that
cluster graph has a unique infinite component almost surely, which
has been proved in the last section, and allows us to get a
polynomially small upper bound for
$\mathbf{P}[\dist(0,\mathscr{C})\geq k]$, which is weaker than
(\ref{e91}).
\begin{proposition}\label{p7}
Consider critical Bernoulli FPP on $\mathbb{T}$.  The following
properties are valid with probability one.
\begin{itemize}
\item[(i)]  There exists an infinite geodesic from 0 to $\infty$.
Moreover, there exists a sequence of disjoint closed circuits
surrounding 0, such that for any infinite geodesic $\gamma$ starting
from 0, each closed site in $\gamma$ except 0 is in one of these
circuits, with different closed sites lying in different circuits.
\item[(ii)] The number of bad sites is finite.
\end{itemize}
\end{proposition}
\begin{proof}
RSW implies that almost surely there are infinitely many open
clusters surrounding 0, denoted by $\mathcal {C}_1,\mathcal
{C}_2,\ldots$ from inside to outside.  Then it is clear that almost
surely, there exists an infinite geodesic from 0 to $\infty$, and
each infinite geodesic $\gamma$ starting from 0 can be represented
by $0\gamma_{0,1}\gamma_1\gamma_{1,2}\gamma_2\ldots$, where for
$i\geq 1$, $\gamma_i$ is a path in $\mathcal {C}_i$ and
$\gamma_{i,i+1}$ is a geodesic between $\mathcal {C}_i$ to $\mathcal
{C}_{i+1}$, and $\gamma_{0,1}$ is a geodesic between 0 and $\mathcal
{C}_1$.  Similarly to item (ii) in Proposition \ref{p3}, for each
$i\geq 1$, there exist $T(\mathcal {C}_i,\mathcal {C}_{i+1})$
disjoint closed circuits surrounding 0 between $\mathcal {C}_i$ and
$\mathcal {C}_{i+1}$, such that each closed site in $\gamma_{i,i+1}$
is in one of these circuits, with different closed sites lying in
different circuits; there exist $T(0,\mathcal {C}_1)-t(0)$ disjoint
closed circuits surrounding 0 between 0 and $\mathcal {C}_1$, such
that each closed site in $\gamma_{0,1}$ is in one of these circuits,
with different closed sites lying in different circuits.  Then the
first item follows immediately.

Let us turn to the proof of the item (ii).  For each site
$v\in\mathbb{V}$ with $|v|\geq 4$, we let
$K=K(v)=\lfloor\log_2|v|\rfloor$, and define the event
\begin{equation*}
\mathcal {F}_v:=\{\exists~r\mbox{ such that $2\leq r\leq 2^K$ and
$\mathcal {A}_6(v;2,r)\cap\mathcal {A}_{(111111)}(v;r,2^K)$
occurs}\}.
\end{equation*}
Note that $\mathcal {A}_6(v;2,2^K)\subset\mathcal {F}_v$ since we
have set $\mathcal {A}_{(111111)}(v;2^K,2^K)=\Omega$.  Assume that
$|v|\geq 4$.  It is easy to see by item (i) in Proposition \ref{p7}
that if $v$ is a bad site, then $\mathcal {F}_v$ occurs.  Suppose
that the event $\mathcal {F}_v$ holds.   By considering the smallest
$r$ satisfying $\mathcal {F}_v$ with $2^i\leq r\leq 2^{i+1}$, there
exist universal constants $\varepsilon_0,C_1,C_2,C_3>0$ such that
\begin{align*}
\mathbf{P}[\mathcal {F}_v]&\leq \sum_{i=1}^{K-1}\mathbf{P}[\mathcal
{A}_6(2,2^i)]\mathbf{P}[\mathcal {A}_{(111111)}(2^{i+1},2^K)]\\
&\leq \sum_{i=1}^{K-1}C_1\mathbf{P}[\mathcal
{A}_6(2,2^i)]\mathbf{P}[\mathcal {A}_6(2^{i+1},2^K)]2^{-\varepsilon_0(K-i-1)}~~\mbox{by Lemma \ref{l4}}\\
&\leq \sum_{i=1}^{K-1}C_2\mathbf{P}[\mathcal
{A}_6(2,2^K)]2^{-\varepsilon_0(K-i-1)}~~\mbox{ by
quasi-multiplicativity}\\
&\leq C_3\mathbf{P}[\mathcal {A}_6(1,|v|)]~~\mbox{by extendability.}
\end{align*}
This together with (\ref{e67}) implies that, there exist
$\varepsilon,C>0$, such that for all sites $v$ with $|v|\geq 4$,
\begin{equation*}
\mathbf{P}[\mbox{$v$ is a bad site}]\leq \mathbf{P}[\mathcal
{F}_v]\leq C_3\mathbf{P}[\mathcal {A}_6(1,|v|)]\leq
C|v|^{-2-\varepsilon},
\end{equation*}
which gives
\begin{equation*}
\mathbf{E}[\mbox{the number of bad sites}]=\sum_{v\in \mathbb{V}}
\mathbf{P}[\mbox{$v$ is a bad site}]<\infty.
\end{equation*}
So the number of bad sites is finite with probability one.
\end{proof}

\begin{proof}[Proof of Theorem \ref{t6}]
By Proposition \ref{p5}, the cluster graph has a unique infinite
component $\mathscr{C}$ almost surely.  Therefore, there is a
constant $p_0\in (0,1)$ such that
\begin{equation*}
\mathbf{P}[\mathcal {C}_0\in \mathscr{C}]\geq p_0.
\end{equation*}
Conditioned on the event $\{\mathcal {C}_0\in \mathscr{C}\}$, it is
clear that almost surely $D(\mathcal {C}_0,\mathcal {C}_n)$ is
finite for all $n\in\mathbb{N}$, and similarly to item (i) in
Proposition \ref{p3}, the first-passage time $T(0,\mathcal {C}_n)$
is equal to the maximal disjoint closed circuits surrounding 0 in
the component of $\mathbb{T}\backslash \mathcal {C}_n$ containing 0.

By RSW, it is easy to see that infinitely many of the events
$\{\exists$ an open cluster surrounding 0 in
$A(2^k,2^{k+1})\}_{k\in\mathbb{N}}$ occur almost surely.  This fact
together with Proposition \ref{p7} implies that with probability one
there exists a random $k_0$ (depending on percolation configuration
$\omega$), such that $A(2^{k_0},2^{k_0+1})$ contains an open cluster
surrounding 0, there exist no bad sites outside $B(2^{k_0})$, and
for all $2^{k_0}\leq m\leq n$,
\begin{equation*}
D(\mathcal {C}_m,\mathcal {C}_n)=T(\mathcal {C}_m,\mathcal {C}_n).
\end{equation*}

Define the event
\begin{equation*}
\mathcal {E}_k:=\{T(0,\mathcal {C}_{2^k})-T(0,\partial B(2^k))\geq
k^{1/3}\}.
\end{equation*}
By Lemma \ref{l10} and RSW, it is standard to prove that there exist
$C_1,C_2>0$, such that for all large $k$,
\begin{align*}
&\mathbf{P}[\mathcal {E}_k]\leq \mathbf{P}[\mathcal
{C}_{2^k}\nsubseteq
A(2^k,2^{k+C_1k^{1/3}})]+\mathbf{P}[S(2^k,2^{k+C_1k^{1/3}})\geq
k^{1/3}]\leq \exp(-C_2k^{1/3}).
\end{align*}
Then by Borel-Cantelli lemma, almost surely only finitely many of
$\mathcal {E}_k$'s occur.  So with probability one there exists a
random $k_1$, such that for all $k\geq k_1$, the event $\mathcal
{E}_k$ does not occur.

The arguments above implies that, conditioned on the event
$\{\mathcal {C}_0\in \mathscr{C}\}$, almost surely for all $k\geq
\max\{k_0,k_1\}$,
\begin{equation}\label{e61}
D(\mathcal {C}_0,\mathcal {C}_{2^k})-T(0,\partial B(2^k))=D(\mathcal
{C}_0,\mathcal {C}_{2^{k_0}})+T(\mathcal {C}_{2^{k_0}},\mathcal
{C}_{2^k})-T(0,\partial B(2^k))\leq 2^{k_0+1}+k^{1/3}.
\end{equation}
It is obvious that conditioned on $\{\mathcal {C}_0\in
\mathscr{C}\}$, almost surely for all $k\geq 1$ and $2^k\leq n\leq
2^{k+1}$,
\begin{equation}\label{e62}
D(\mathcal {C}_0,\mathcal {C}_{2^k})\leq D(\mathcal {C}_0,\mathcal
{C}_n)\leq D(\mathcal {C}_0,\mathcal {C}_{2^{k+1}}).
\end{equation}
Combining (\ref{e61}), (\ref{e62}) and (\ref{t41}), we obtain
Theorem \ref{t6}.
\end{proof}

\section{Appendix: Proof of (\ref{e42})}
\begin{proof}[Proof of (\ref{e42})]
The proof is essentially the same as the proof of (2.24) in
\cite{KZ97}.  Recall that for $j\geq 0$,
\begin{equation*}
\widetilde{\Delta}_j(\omega):=\mathbf{E}_p[T(0,\partial
B(L(p)))\mid\widetilde{\mathscr{F}}_j]-\mathbf{E}_p[T(0,\partial
B(L(p)))\mid\widetilde{\mathscr{F}}_{j-1}].
\end{equation*}
First, let us show that for all $j\geq 0$,
\begin{align}
\widetilde{\Delta}_j(\omega)=&T(\widetilde{\mathcal
{C}}_{j-1}(\omega),\widetilde{\mathcal
{C}}_j(\omega))(\omega)+\mathbf{E}_p'[T(\widetilde{\mathcal
{C}}_j(\omega),\partial B(L(p)))(\omega')]\nonumber\\
&-\mathbf{E}_p'[T(\widetilde{\mathcal {C}}_{j-1}(\omega),\partial
B(L(p)))(\omega')].\label{e100}
\end{align}

For $j\geq 1$, observe that
\begin{equation}\label{e101}
T(0,\partial B(L(p)))=T(0,\widetilde{\mathcal
{C}}_{j-1})+T(\widetilde{\mathcal {C}}_{j-1},\widetilde{\mathcal
{C}}_j)+T(\widetilde{\mathcal {C}}_j,\partial B(L(p))).
\end{equation}
Note that $T(0,\widetilde{\mathcal {C}}_{j-1})$ depends only on the
sites in $\overline{\widetilde{\mathcal {C}}_{j-1}}$, and
$T(\widetilde{\mathcal {C}}_{j-1},\widetilde{\mathcal {C}}_j)$
depends only on the sites in $\overline{\widetilde{\mathcal
{C}}_j}\backslash\{\mbox{interior of }\widetilde{\mathcal
{C}}_{j-1}\}$.  So $T(0,\widetilde{\mathcal {C}}_{j-1})$ is
$\widetilde{\mathscr{F}}_{j-1}$-measurable, and
$T(\widetilde{\mathcal {C}}_{j-1},\widetilde{\mathcal {C}}_j)$ is
$\widetilde{\mathscr{F}}_j$-measurable.  Observe that once
$\widetilde{\mathcal {C}}_j$ is fixed, $T(\widetilde{\mathcal
{C}}_j,\partial B(L(p)))$ depends only on sites which lie in
$B(L(p))\backslash \overline{\widetilde{\mathcal {C}}_j}$; given the
configuration of the sites in $\overline{\widetilde{\mathcal
{C}}_j}$, the sites in $B(L(p))\backslash
\overline{\widetilde{\mathcal {C}}_j}$ are conditionally independent
of this configuration.  Therefore,
\begin{equation}\label{e102}
\mathbf{E}_p[T(\widetilde{\mathcal {C}}_j,\partial B(L(p)))\mid
\widetilde{\mathscr{F}}_j](\omega)=\mathbf{E}_p'[T(\widetilde{\mathcal
{C}}_j(\omega),\partial B(L(p)))(\omega')].
\end{equation}
This together with (\ref{e101}) and the preceding remarks gives
\begin{equation}\label{e103}
\mathbf{E}_p[T(0,\partial B(L(p)))\mid
\widetilde{\mathscr{F}}_j](\omega)=T(0,\widetilde{\mathcal
{C}}_{j-1})(\omega)+T(\widetilde{\mathcal
{C}}_{j-1},\widetilde{\mathcal
{C}}_j)(\omega)+\mathbf{E}_p'[T(\widetilde{\mathcal
{C}}_j(\omega),\partial B(L(p)))(\omega')].
\end{equation}
Similarly, we have
\begin{equation}\label{e104}
\mathbf{E}_p[T(0,\partial B(L(p)))\mid
\widetilde{\mathscr{F}}_{j-1}](\omega)=T(0,\widetilde{\mathcal
{C}}_{j-1})(\omega)+\mathbf{E}_p'[T(\widetilde{\mathcal
{C}}_{j-1}(\omega),\partial B(L(p)))(\omega')].
\end{equation}
Then for $j\geq 1$, (\ref{e100}) follows by subtracting (\ref{e104})
from (\ref{e103}).

It remains to show (\ref{e100}) for $j=0$.  Observe that
\begin{equation*}
T(0,\partial B(L(p)))=T(0,\widetilde{\mathcal
{C}}_1)+T(\widetilde{\mathcal {C}}_1,\partial B(L(p))).
\end{equation*}
Essentially the same proof as above shows that
\begin{equation}\label{e105}
\mathbf{E}_p[T(0,\partial B(L(p)))\mid
\widetilde{\mathscr{F}}_1](\omega)=T(0,\widetilde{\mathcal
{C}}_1)(\omega)+\mathbf{E}_p'[T(\widetilde{\mathcal
{C}}_1(\omega),\partial B(L(p)))(\omega')].
\end{equation}
It is clear that
\begin{equation}\label{e106}
\mathbf{E}_p[T(0,\partial B(L(p)))]=\mathbf{E}_p'[T(0,\partial
B(L(p)))].
\end{equation}
Then for $j=0$, (\ref{e100}) follows by subtracting (\ref{e106})
from (\ref{e105}).

Finally, note that for $j\geq 0$,
\begin{align}
&T(\widetilde{\mathcal {C}}_j(\omega),\partial B(L(p)))(\omega')=T(\widetilde{\mathcal {C}}_j(\omega),\widetilde{\mathcal {C}}_{l(j,\omega,\omega')}(\omega'))(\omega')+T(\widetilde{\mathcal {C}}_{l(j,\omega,\omega')}(\omega'),\partial B(L(p)))(\omega'),\label{e107}\\
&T(\widetilde{\mathcal {C}}_{j-1}(\omega),\partial
B(L(p)))(\omega')=T(\widetilde{\mathcal
{C}}_{j-1}(\omega),\widetilde{\mathcal
{C}}_{l(j,\omega,\omega')}(\omega'))(\omega')+T(\widetilde{\mathcal
{C}}_{l(j,\omega,\omega')}(\omega'),\partial
B(L(p)))(\omega').\label{e108}
\end{align}
Substitution of \ref{e107} and (\ref{e108}) into the right hand side
of (\ref{e100}) yields (\ref{e42}).
\end{proof}

\section*{Acknowledgements}
The author thanks Geoffrey Grimmett for an invitation to the
Statistical Laboratory in Cambridge University, and thanks the
hospitality of the Laboratory, where this work was completed.  The
author also thanks an anonymous referee for a detailed report that
contributed to a better presentation of this manuscript.  The author
was supported by the National Natural Science Foundation of China
(No. 11601505), an NSFC grant No. 11688101 and the Key Laboratory of
Random Complex Structures and Data Science, CAS (No. 2008DP173182).

\end{document}